\documentclass[11pt,a4paper,fleqn]{article}
\usepackage{amsmath}
\usepackage{amssymb}
\usepackage{amscd}
\usepackage{color}
\usepackage[dvips]{graphicx}
\usepackage{anysize}
\marginsize{1 in}{1 in}{0 in}{1.5 in} 

\newtheorem{thm}{Theorem}

\newtheorem{cor}[thm]{Corollary}

\newtheorem{defi}[thm]{Definition}
\newtheorem{ex}[thm]{Example}
\newtheorem{lem}[thm]{Lemma}
\newtheorem{prop}[thm]{Proposition}
\newtheorem{rem}[thm]{Remark}
\newenvironment{proof}{{\bf Proof. }}{\hfill$\rule{1ex}{1ex}$\par\medskip}
\begin{document}

\title{Nuclear properties of loop extensions}

\author{P\'eter T. Nagy
\date{Institute of Applied Mathematics, 
		\'Obuda University,\\ H-1034 Budapest, B\'ecsi \'ut 96/b, Hungary;\\ {e--mail:}~{nagy.peter@nik.uni-obuda.hu}}}
\footnotetext{\par}
\maketitle
\hspace{4cm} {\Large In memory of Karl Strambach}

\begin{abstract}
The objectives of this paper is to give a systematic investigation of extension theory
of loops. A loop extension is (left, right or middle) nuclear, if the kernel of the extension
consists of elements associating (from left, right or middle) with all elements of the
loop. It turns out that the natural non-associative generalizations of the Schreier's
theory of group extensions can be characterized by different types of nuclear
properties. Our loop constructions are illustrated by rich families of examples in
important loop classes.
\end{abstract}


\section{Introduction}

Over the past two decades, a non-associative extension theory of binary systems has attracted considerable attention, especially from the viewpoints of construction and investigation of quasigroups and loops with special features (e.g. \cite{BDRX}, \cite{BX}, \cite {DrV}, \cite{FiSt1}, \cite{FiSt2}, \cite{GR}, \cite{KN}, \cite{KJ}, \cite{npl}, \cite{nsl}).  Among these, a non-associative generalization of the group theoretical Schreier extension (cf. \cite{NS}) seems to be the simplest and serves as a prototype. 
A loop $L$ is an extension of a normal subloop $N$ by the loop $K$ if $K$ is isomorphic to the factor loop $L/N$, the normal subloop is called the kernel subloop of the extension. Equivalently, there is a short exact sequence of loops 
\[1\to M\to L\to K\to 1,\]
such that $M$ is isomorphic to the normal subloop $N$. 
Extension theory of loops deals with classification of all possible extensions of loops and studies their properties. The 
related problems of group theory are completely solved by the Schreier theory of group extensions, cf. \cite{Sch1}, \cite{Sch2}, \cite{Kur}, 
Chapter XII, \S 48-49, pp. 121--131 in \cite{Kur}. An interesting property of extensions is the (left, right or middle) nuclear property, the kernel of a (left, right or middle) nuclear extension consists of elements associating (from left, right or middle) with all elements of the loop. The related discussions are applied to the construction and classification of new families of loops (e.g. \cite{DrV}, \cite{HJ}, \cite{Jedl}, \cite{HJKV}, \cite{KKPhV}, \cite{KPhV}, \cite{NSt1}, \cite{NSt2}).\\
Isomorphism classes of group extensions can be described with the cohomology classes assigned to the extension. Hence for the extension theory of groups homological algebra yields a natural framework, and considerable efforts have been made to extend the cohomological method even in the non-associative case. S. Eilenberg, S. Maclane \cite{EMC} and J. P. Lafuante \cite{Laf} introduced and studied a class of non-associative extension of groups by groups, called loop prolongation.  They observed that cohomological cocycles can measure the non-associativity of loop extensions, moreover a cohomological interpretation of the equivalence of extensions was formulated. K. W. Johnson, C. R. Leedham-Green \cite{JLG}, J. D. H. Smith \cite{JDHSm}, and R. Lal, B. K. Sharma \cite{LaSh} initiated to develop the theory of extensions to more general multiplicative structures. The methods of cohomology theory had successful applications to classification problems of central extensions of abelian groups by loops and to extensions of Bol and Moufang loops by G. P. Nagy, P. Vojt\v{e}chovsk\'y, D. Daly, \cite{NaGP}, \cite{NaVo}, \cite{DaVo}, N. Nishig\^{o}ri \cite{Nish1}, \cite{Nish2} and R. Jimenez, Q. M. Mel\'endez \cite{JiMe}. \\
However, to interpret the constructions of loop extensions in cohomology theory requires some structural theorems for the considered classes of loops. In the general case of nonassociative multiplication, the cohomological approaches encounter difficulties and limitations, hence the study requires more technical and combinatorial tools. The present contribution is therefore dedicated to the systematic investigation of the basic constructions of nonassociative extension theory using direct computational methods.\\
We studied in our joint work with K. Strambach \cite{NS} the Schreier extensions, i.e. extensions of groups by loops defined by the same formula as in group theory, and obtained characterizations and constructions of examples for interesting subclasses of such loops. It turned out that in this case the extension is middle and right nuclear. The non-associative extension theory of Schreier type is investigated in a broader context in our papers with I. Stuhl \cite{NSt1} and \cite{NSt2}, namely we gave characterizations of some right nuclear extensions of groups by quasigroups with right identity element.\\
In 1944 A. A. Albert in \cite{Alb}, (Theorem 6, pp. 406-407) and R. H. Bruck in \cite{Bru1}, (Theorem 10\,B, pp. 166-168) initiated and thoroughly investigated the most general constructions of loop extensions. It turned out that Bruck's construction of loop extensions has many more degrees of freedom than in Schreier's theory of group extensions, namely that the multiplication function between different non-trivial left cosets of the kernel subgroup can be arbitrarily prescribed, (cf. \cite{Che}, pp. 35--43). Albert's description of general loop extensions yields a canonical form of the extended loop expressed with the help of a left transversal to the normal subgroup, any extension can be represented in Albert's form.\\
The objectives of this paper are to give a systematic investigation of nuclear properties of extensions, to find algebraic characterizations of their different types and to examine the limits of the non-associative generalization of theory of extensions. 
The first section contains the introduction and the necessary definitions and the basic constructions. Section 2 is devoted to the general theory of loop extensions and to the discussion of their nuclear properties. \S 2.1 examines the relationship between the Bruck's and Albert's description of extensions. \S 2.2 contains a discussion of the internal determination of an Albert extension by a triple of a normal subgroup, a loop isomorphism to the factor loop and a left transversal to the kernel subgroup. \S 2.3 is devoted to the classification of extensions of order $6$ and to the description of their nuclear properties. \S 2.4 gives a characterization of right nuclear extensions in Albert's form, in this case the multiplication function between cosets of the kernel subgroup is a left multiplication in the kernel loop. \S 2.5 contains a construction of right nuclear but not nuclear extensions having left or right inverse property.
In Section 3 we investigate right and middle nuclear extensions, called the Schreier extensions, which are the closest analogues of the group theoretical case. \S 3.1 is devoted to the construction of large classes of non-nuclear Schreier extensions satisfying the right Bol identity, using an extension method of Bol loops developed in \cite{nsl} for classification of compact Bol loops. In \S 3.2 a discussion of internal constructions of Schreier extensions is given. In this study, the conjugation of the kernel subgroup by elements of the loop are an interesting tool. 
\S 3.3 is devoted to the investigation of automorphism-free and factor-free Schreier extensions by analogy in group theory and to the discussion of changing an extension by changing the underlying isomorphism or the left transversal of the decomposition. 
\S In 3.4 we investigate the relationship between different Schreier extensions of a loop.

\subsection*{Preliminaries}

A {\it quasigroup} $L$ is a set with a multiplication map $(x,y)\mapsto x\cdot y:L\times L\to L$ such that for each $x\in L$ the  {\it left translations}  $\lambda_x :L\to L$,  $\lambda_xy = xy$,  and the {\it right translations}  $\rho_x : L\to L$, $\rho_xy =  yx$, are bijective maps. The left and right division operations on $L$ are defined by the maps $(x,y)\mapsto x\backslash y = \lambda_x^{-1}y$, respectively  $(x,y)\mapsto x/y = \rho_y^{-1}x$,  
$x,y\in L$. An element $e\in L$ is called left (right) identity if it satisfies $e\cdot x = x$ ($x\cdot e = x$) for any $x\in L$. A left and right identity is called identity element. A quasigroup $L$ is a {\it loop} if it has an identity element. The {\it right inner mappings} of a loop $L$ are the maps $\rho^{-1}_{yx} \rho_x \rho_y: L\to L$, $x,y\in L$, the group generated by the right inner mappings is the {\it right inner mapping group}. The automorphism group of of a loop $L$ is denoted by $\mathrm{Aut}(L)$.\\
We will reduce the use of parentheses by the following convention: juxtaposition will denote multiplication, the division operations are less binding than juxtaposition, and the multiplication symbol is less binding than the divisions. For instance the expression $xy / u \cdot v \backslash w$ is
a short form of $((x\cdot y)/ u)\cdot(v\backslash w)$.\\
The subgroups 
\[\begin{split}
&\mathcal{N}_l(L)\; = \{u\in L; \;ux\cdot y = u\cdot xy, \;\; x, y\in L\},\\
&\mathcal{N}_r(L)\, = \{u\in L; \;xy\cdot u = x\cdot yu, \;\; x, y\in L\},\\
&\mathcal{N}_m(L) = \{u\in L; \;xu\cdot y = x\cdot uy, \;\; x, y\in L\}\end{split}\]
of a loop $L$ are the {\it left, right,} respectively  {\it middle nuclei}, 
and $\mathcal{N}(L)=\mathcal{N}_l(L)\cap \mathcal{N}_r(L)\cap \mathcal{N}_m(L)$ is  the {\it nucleus} of $L$. A non-associative loop is called {\it proper loop}. If a subloop $K\subset L$ is contained in the (left, right, respectively middle) nucleus of $L$, the it is {\it (left, right,} respectively  {\it middle) nuclear}, in this case $K$ is necessarily a group. 
For a subloop $K\subset L$ the subset $\mathcal{C}_L(K) = \{u\in L; \; u\cdot x = x\cdot u, \;\; x\in L\}$ is the {\it commutant} of $K$. \\
A loop $L$ satisfies the {\it left}, respectively the {\it right inverse property} 
if there exists a bijection $x\mapsto x^{-1}:L\to L$ such that $ x^{-1}\cdot xy = y$, respectively 
$yx\cdot x^{-1}= y$ holds for all $x, y\in L$. A loop with {\it left} and {\it right inverse property} has {\it  inverse property}. The {\it left alternative}, respectively {\it right alternative} property of $L$ is defined by the identity 
$x\cdot xy = x^2 y$, respectively $yx\cdot x = y x^2$,\; $L$ is {\it flexible} if $ x\cdot yx = xy\cdot x$ for all $x, y\in L$. A {\it left}, respectively {\it right Bol loop} is defined by the identity $(x\cdot yx)z = x(y\cdot xz)$, respectively  $z(xy\cdot x)   = (zx\cdot y)x$. Any left (right) Bol loop has the left (right) alternative and the left (right) inverse property. \\
A subloop $N\subset L$ is {\it normal} if it is the kernel of a homomorphism of $L$. The {\it factor loop} $L/N$ is the loop induced on the set of left cosets of the normal subloop $N$. \\
A subset $\Sigma\subset L$ is a {\it left (right) transversal}  
to a subgroup $G\subset L$ if $\Sigma$ consists of representatives of left (right) cosets of $G$ such that $e\in\Sigma$.  

\section{Albert extensions}

\subsection{Bruck and Albert extensions}

Let $K$ and $N$ be loops and denote by small Greek, respectively Latin letters the elements of $K$, respectively of $N$ and by $\epsilon\in K$ and $e\in G$ their identity elements. 
\begin{defi} \label{extension}
	A loop $L$ is an \emph{extension} of the normal proper subloop $N$ by the loop $K$, if there is an isomorphism $\kappa$ of $K$ onto $L/N$. We say that $N\lhd L$ is the kernel and $\kappa: K\to L/N$ is the underlying isomorphism of the extension.\\
	An extension is (\emph{left, right,} respectively \emph{middle}) \emph{nuclear}, if the kernel subloop is (\emph{left, right}, 
	respectively \emph{middle}) \emph{nuclear} subgroup of $L$. 
\end{defi}
A. A. Albert formulated a principal description of extensions of loops by loops what he called \emph{crossed extension of loops by loops}. R. H. Bruck investigating the  general extensions of quasigroups by quasigroups gave a somewhat more general construction for extensions of loops by loops, which can be summarized as follows:
\begin{lem} \label{Bruck} (R. H. Bruck, \cite{Bru1}, Theorem 10B, p. 168 and \cite{Bru2}, p. 779.) 
Assume that a loop $L$ with identity element $e\in L$ is an extension of a normal proper subloop $N$ by a loop $K$. Then $L$ is isomorphic to the loop $K \times^\triangledown N$ defined on $K\times N$ by the multiplication  
\begin{equation} \label{brumul} (\alpha,a)(\beta,b)= (\alpha\beta,\triangledown_{\alpha,\beta}(a,b))),\quad  \alpha,\beta\in K,\quad a,b\in N,\end{equation}  
where $(a,b)\mapsto\triangledown_{\alpha,\beta}(a,b)$, $\alpha,\beta\in K$ are quasigroup multiplications on $N$ satisfying the conditions 	
\begin{enumerate}
\item[\emph{(i)}] $\triangledown_{\epsilon,\epsilon}$ coincides with the multiplication of $N$,  
\item[\emph{(ii)}] $e\in N$ is a left identity element of the multiplications  $\triangledown_{\epsilon,\alpha}$, 
\item[\emph{(iii)}] $e\in N$ is a right identity element of the multiplications  $\triangledown_{\alpha,\epsilon}$, $\alpha \in K$.
\end{enumerate} 
Conversely, if for any $\alpha,\beta\in K$ there is given a quasigroup multiplication $\triangledown_{\alpha,\beta}$ on a loop $N$ satisfying the conditions \emph{(i), 
(ii), (iii)}, then the multiplication  (\ref{brumul})  
determines a loop $K \times^\triangledown N$ which is an extension of the subloop $\{\epsilon\}\times N$ by the loop $K$.
\end{lem}
\begin{defi} The loop $K \times^\triangledown N$ defined on $K\times N$ determined by the quasigroup multiplication $\triangledown_{\alpha,\beta}(a,b)$, $\alpha,\beta\in K$ satisfying the conditions \emph{(i)}, \emph{(ii)}, \emph{(iii)} of Lemma \ref{Bruck} is called the \emph{Bruck extension} of the loop $N$ by the loop $K$. 
\end{defi}
For the determination of crossed extensions A. A. Albert gave the following construction: 
\begin{lem}\label{Albert} (A. A. Albert, \cite{Alb}, Theorem 6, pp. 406-407.) Let $\Sigma$ be a left transversal to a normal subloop $N$ in a loop $L$ and let $\kappa:K\to L/N$ be an isomorphism of the loop $K$ to the factor loop $L/N$. If $l_\sigma$ denotes the unique element of the intersection $\Sigma\cap\kappa(\sigma)$, $\sigma\in K$, then the function  $\triangledown: K\times K\times N\times N\to N$ defined by 
\begin{equation} \label{triangledown} \triangledown_{\xi,\eta}(x,y) = l_{\xi\eta}\backslash (l_\xi x \cdot l_\eta y),\; \xi,\eta\in K, \;x,y\in N,  \end{equation} 
determines a Bruck extension $K \times^\triangledown N$ of the loop $N$ by the loop $K$ and the map $\Phi:(\xi,x) \mapsto l_\xi x$ is an isomorphism $\Phi:K \times^\triangledown N\to L$.
\end{lem}
\begin{proof} If $\pi:L\to L/N$ is the canonical homomorphism then we have $\pi\left(l_{\xi\eta}\backslash (l_\xi x \cdot l_\eta y)\right) = N$ 
	and hence $\triangledown_{\xi,\eta}(x,y)$ is contained in $N$. The conditions $\{i\}$, $\{ii\}$, $\{iii\}$ of Lemma \ref{Bruck} are routinely verified.
One has $\Phi((\xi,x)(\eta,y)) = l_\xi x\cdot l_\eta y = \Phi(\xi,x)\Phi(\eta,y)$  
and hence the assertion follows.
\end{proof}
\begin{rem}\label{leftraiii} Let $\Sigma$ be a left transversal to $N$ in $L$
	 and $K \times^\triangledown N$ the Bruck extension determined by the quasigroup multiplications (\ref{triangledown}). Then  
\begin{enumerate} 
	\item[\emph{(iii$^\prime$)}] $e\in N$ is the identity element of the multiplications  $\triangledown_{\alpha,\epsilon}$, $\alpha \in K$. 
\end{enumerate} 
\end{rem}
Clearly, (iii) of Lemma \ref{Bruck} is a consequence of (iii$^\prime$) of Remark \ref{leftraiii}. 
\begin{defi} \label{Bruckex} A Bruck extension $K \times^\triangledown N$ is called an \emph{Albert extension} of $N$ by $K$ if the quasigroup multiplications $\triangledown_{\alpha,\beta}:N\times N\to N$, $\alpha,\beta\in K$ satisfy 	
	\begin{enumerate}
		\item[\emph{(a)}] $\triangledown_{\epsilon,\epsilon}$ coincides with the multiplication of $N$,  
		\item[\emph{(b)}] $e\in N$ is a left identity element of  $\triangledown_{\epsilon,\alpha}$, $\alpha \in K$, 
		\item[\emph{(c)}] $e\in N$ is the identity element of $\triangledown_{\alpha,\epsilon}$, $\alpha \in K$.
	\end{enumerate} 
\end{defi}
It follows from Lemma \ref{Albert} and Remark \ref{leftraiii}, that any Bruck extension $K \times^\triangledown N$ is isomorphic to an Albert extension. Now, we construct such an isomorphism. 
\begin{thm} \label{brualb}
Let  $K \times^\triangledown N$ be a Bruck extension and $\varphi_\alpha:N\to N$ the maps $\varphi_\alpha:x\mapsto \triangledown_{\alpha,\epsilon}(e,x)$, $\alpha\in K$. Then the map $\Phi:(\xi,x)\mapsto(\xi,\varphi_\xi(x))$ is an isomorphism of the Albert extension $K\times^{\triangledown^*}N$ determined by the function $\triangledown^*: K\times K\times N\times N\to N$, 
\[\triangledown^*_{\alpha,\beta}(a,b) = \varphi_{\alpha\beta}^{-1}\triangledown_{\alpha,\beta}(\varphi_\alpha(a),\varphi_\beta(b)), \quad \alpha,\beta\in K,\] 
onto the Bruck extension $K \times^\triangledown N$.
Moreover,  the isomorphism $\Phi$ induces the identity map on the normal subgroup $\{\epsilon\}\times N$. 
\end{thm}
\begin{proof} We have \[\triangledown^*_{\epsilon,\alpha}(e,b) = \varphi_{\alpha}^{-1}\triangledown_{\epsilon,\alpha}(e,\varphi_\beta(b)) = b\quad\text{for any}\quad\alpha\in K,\; b\in N,\] hence $e\in N$ is a left identity element of  $\triangledown_{\epsilon,\alpha}$. Moreover $e\in N$ is the identity element of $\triangledown_{\alpha,\epsilon}$ for any $\alpha \in K$, since the identities 
\[\triangledown^*_{\alpha,\epsilon}(e,b) = \varphi_{\alpha}^{-1}\triangledown_{\alpha,\epsilon}(e,b) = b \quad\text{and}\quad \triangledown^*_{\alpha,\epsilon}(b,e) = \varphi_{\alpha}^{-1}\triangledown_{\alpha,\epsilon}(\varphi_\alpha(b),e) = 
\varphi_{\alpha}^{-1}\varphi_\alpha(b) = b\] 
follow from $\varphi_\alpha(e) = e$, $\varphi_\epsilon(b) = b$ and from the right identity property of $e\in N$ with respect to the multiplication $\triangledown_{\alpha,\epsilon}$. 
Hence $K\times^{\triangledown^*}N$ is an Albert extension. Since $\triangledown_{\epsilon,\epsilon}(e,x) = x$, one has $\Phi(\epsilon,x) = (\epsilon,x)$ for any $x\in N$. The multiplication $(\alpha,a)*(\beta,b)$ in $K \times^{\triangledown^*}N$ satisfies  
\[\Phi(\alpha,a)\cdot\Phi(\beta,b) = \left(\alpha\beta,\triangledown_{\alpha,\beta}(\varphi_\alpha(a),\varphi_\beta(b))\right) = \left(\alpha\beta,\varphi_{\alpha,\beta}(\triangledown^*_{\alpha,\beta}(a,b))\right) = \Phi((\alpha,a)*(\beta,b)),\] 
consequently the  map $(\xi,x)\mapsto(\xi,\varphi(x))$ is an isomorphism from $K\times^{\triangledown^*}N$ to  $K \times^\triangledown N$.  
\end{proof}
\begin{cor} \label{coincide}
If $\Sigma = \{(\xi,e),\; \xi\in K\}$ is the left transversal to the subgroup $\{\epsilon\}\times N$ in an Albert extension $K\times^\triangledown N$ then the multiplications defined by (\ref{triangledown}) with respect to $\Sigma$ coincide with the multiplications $\triangledown_{\alpha,\beta}$ for any $\alpha,\beta\in K$.
\end{cor}
\begin{proof} If $l_\xi = (\xi,e)$ then we have $l_{\xi\eta}\backslash (l_\xi x \cdot l_\eta y) = (\xi,e)\backslash ((\xi,x)(\eta,y)) = (\epsilon,\triangledown_{\xi,\eta}(x,y))$.
\end{proof}

\subsection{Albert decompositions}

Let $L$ be a loop, $N$ a normal subloop of $L$ and let $K\times^\triangledown N$ be an Albert extension of $N$ by $K$. 
\begin{lem} If the map $\mathcal{F}:K\times^\triangledown N\to L$ is an isomorphism satisfying $\mathcal{F}(\epsilon,t) = t$ for any  $t\in N$, then the induced map $$\sigma\mapsto\mathcal{F}(\sigma,e)G:K\to L/N$$ is an isomorphism, too. 
\end{lem}
\begin{proof}  The assertion follows from the observation that a loop isomorphism   $\tau:M\to N$ induces an isomorphism $M/H\to N/\tau(H)$, where $H\lhd M$ is a normal subloop.
\end{proof}
\begin{defi} An \emph{Albert decomposition of a loop $L$ with respect to} its normal subloop $N$ is an isomorphism $\mathcal{F}$ of an Albert extension $K\times^\triangledown N$ onto $L$ satisfying $\mathcal{F}(\epsilon,t) = t$ for any  $t\in N$. The  map  $\sigma\mapsto\mathcal{F}(\sigma,e)G:K\to L/N$ is the  \emph{underlying isomorphism} of the Albert decomposition $\mathcal{F}$.\\ 
	A \emph{data triple $(N,\kappa,\Sigma)$ of an Albert decomposition} of the loop  $L$ consists of 
	\begin{enumerate}
		\item[(\emph{i)}] a normal subloop $N$ of $L$,
		\item[(\emph{ii)}] an isomorphism $\kappa$ of a loop $K$ onto $L/N$, 
		\item[(\emph{iii)}] a left transversal $\Sigma$ to $N$ in $L$.
	\end{enumerate}
\end{defi}
For a data triple $(N,\kappa,\Sigma)$ of an Albert decomposition of $L$ we define the maps 
\begin{equation}\label{lTfAlb} l_\sigma = \Sigma\cap\kappa(\sigma),  \quad \text{and}\quad \triangledown_{\sigma,\tau}(s,t) = 
l_{\sigma\tau}\backslash(l_\sigma s\cdot l_\tau t),\quad \sigma,\tau\in K,\; s,t\in N. \end{equation}
\begin{defi} The loop determined by the maps \emph{(\ref{lTfAlb})} on $K\times N$ will be called the \emph{Albert extension determined by the data triple} $(N,\kappa,\Sigma)$ and denoted by $\mathcal{L}(N,\kappa,\Sigma)$. 
\end{defi} 
\begin{thm} \label{sectalb} Let $\mathcal{L}(N,\kappa,\Sigma)$ be the Albert extension determined by the data triple $(N,\kappa,\Sigma)$. The map $\mathcal{F}:\mathcal{L}(N,\kappa,\Sigma)\to L$ defined by $\mathcal{F}(\sigma,s) = l_\sigma s$ is an Albert decomposition of $L$ with respect to $N$ with the underlying isomorphism $\kappa:K\to L/N$.\\
	Conversely, let $\mathcal{F}:K\times^\triangledown N\to L$ be an Albert decomposition of $L$ with respect to $N$. Define the map $\kappa$ by $\sigma\mapsto\mathcal{F}(\sigma,e)N$ and the transversal $\Sigma$ by  $\{\mathcal{F}(\sigma,e); \;\sigma\in K\}$. Then $(N,\kappa,\Sigma)$ is a data triple of an Albert decomposition of $L$ and $K\times^\triangledown N = \mathcal{L}(N,\kappa,\Sigma)$.
\end{thm}
\begin{proof} An element $x\in L$ can be uniquely decomposed as a 
	product $x = l_\sigma s$ with $l_\sigma\in\Sigma$, $s\in N$. The bijective map $\mathcal{F}:K\times G\to L$ defined  by $\mathcal{F}(\sigma,s) = l_\sigma s$ satisfies $\mathcal{F}(\epsilon,t) = t$ for any  $t\in N$.
	We have  
	$\mathcal{F}(\sigma,s)\mathcal{F}(\tau,t) = l_{\sigma}s\cdot  l_{\tau}t = 
	\mathcal{F}(\sigma\tau,l_{\sigma\tau}\backslash(l_\sigma s\cdot l_\tau t))$, 
	hence the map $\mathcal{F}$ is an isomorphism and  $\mathcal{F}(\sigma,e)G = \kappa(\sigma)$. Consequently  $\mathcal{F}:\mathcal{L}(N,\kappa,\Sigma)\to L$ is an Albert decomposition of $L$ and $\kappa$ is the underlying isomorphism. \\
	Conversely, let $\mathcal{F}:K\times^\triangledown N\to L$ be an Albert  decomposition of $L$. The set $\Sigma$ is a left transversal to $N$ in $L$ and $\mathcal{F}(\sigma,s) = l_\sigma s$ for any 
	$\sigma\in K$ and $s\in N$. Clearly, $(N,\kappa,\Sigma)$ is a data triple of an Albert decomposition of $L$ and the equation $K\times^\triangledown N = \mathcal{L}(N,\kappa,\Sigma)$ follows from Corollary \ref{coincide}. 
\end{proof}

\begin{lem} \label{exten} For any Albert extension $K\times^\triangledown N$ there exists an extension $L$ of 
	$N$ by $K$ and an isomorphism $\mathcal{F}:K\times^\triangledown N$ which is an Albert decomposition of $L$ with respect to $N$.
\end{lem}
\begin{proof} We replace in $K\times N$ the elements $(\epsilon,t)\in\{\epsilon\}\times N$ by the corresponding elements $t\in N$ and define a loop $L$ 
	on the set $\big((K\times N)\setminus \{\epsilon\}\times N\big)\cup N$ in such a way, that the map $(\epsilon,t)\mapsto t$, $t\in N$, together with 
	the identity map on $(K\times N)\setminus \{\epsilon\}\times N$ is an isomorphism $\mathcal{F}:K\times^\triangledown N\to L$. Then $\mathcal{F}:K\times^\triangledown N\to L$ is an Albert decomposition of $L$ with respect to $N$. \end{proof}

\subsection{Albert extensions of order $6$}

\subsubsection*{Extensions of $C_2$ by $C_3$}
Since the loops of order $2$ or $3$ are cyclic groups, the Albert extensions of order $6$ are extensions of a cyclic group by a cyclic group. We investigate the Albert extension loop $L$ of a normal subgroup $C_2 =\{0,1\}$ by the group $C_3 =\{0,1,2\}$ using the construction described in Definition \ref{Albert}. A quasigroup of order two is either $C_2$ or it is given by the multiplication $\triangledown(a,b) = a + b + 1$, $a,b\in C_2$. Hence we have  $\triangledown_{\alpha,\beta}(a,b) = a + b + f(\alpha,\beta)$, where $f:C_3\times C_3\to C_2$ is a function (loop cocycle) satisfying $f(\alpha,0) = f(0,\beta) = 0$, $\alpha,\beta\in C_3$. Let $\mathcal{L}_f$ denote the loop defined by the multiplication  
\[(\alpha,a)(\beta,b) = (\alpha + \beta,\,f(\alpha,\beta) + a + b),\quad  \alpha,
\beta\in C_3,\quad a,b\in C_2\] 
on $C_3\times C_2$. According to Theorem 2.1. in \cite{NaVo} $\{0\}\times C_3$ is contained in the center of $C_3$, i.e. all elements of $\{0\}\times C_3$ commute and associate with all
other elements of $\mathcal{L}_f$, hence the extension is nuclear. Since 
\[(\alpha,a)(\beta,b)\cdot(\gamma,c) = (\alpha + \beta + \gamma,f(\alpha + \beta,\gamma) + f(\alpha,\beta) + a + b + c)\] 
and 
\[(\alpha,a)\cdot (\beta,b)(\gamma,c) = (\alpha + \beta + \gamma,f(\alpha,\beta + \gamma) + a + f(\beta,\gamma) + b + c),\]
the multiplication of $\mathcal{L}_f$ is associative if and only if $f(\alpha + \beta,\gamma) + f(\alpha,\beta) = f(\alpha, \beta +\gamma) + f(\beta,\gamma)$ for any $ \alpha,
\beta,\gamma\in C_3$. Investigating all of the cases $\alpha, \beta, \gamma\in \{1,2\}$, we find that $\mathcal{L}_f$ is a group if and only if 
\begin{equation} \label{f1221} f(1,2) = f(2,1) = f(1,1) + f(2,2).\end{equation} 
 Let $f_i$ be the function defined on the pairs $(1,1)$, $(1,2)$, $(2,1)$, $(2,2)$ with the binary digits of $i = 0,\dots ,15$, i.e. $i = f_i(1,1)2^3 + f_i(1,2)2^2 + f_i(2,1)2 + f_i(2,2)$. 
 Any Albert extension of $C_2$ by $C_3$ is isomorphic to one of the loops $\mathcal{L}_{f_i}$, $i = 0,\dots ,15$.
Let $\mathcal{L}_{f^\prime}$ be the loop determined by the function $f^\prime:C_3\times C_3\to C_2$,  $f^\prime(\alpha,\beta) = f(2\alpha,2\beta)$ and consider the bijective map $\varphi:\mathcal{L}_{f}\to\mathcal{L}_{f^\prime}$ given by $\varphi(\alpha,a) = (2\alpha,a)$. Since 
 \[\varphi(\alpha,a)\varphi(\beta,b) = (2\alpha + 2\beta,\,f^\prime(2\alpha,2\beta) + a + b) = 
 \varphi((\alpha,a)(\beta,b)) = (2\alpha + 2\beta,\,f(\alpha,\beta) + a + b),\] 
 the mapping $\varphi$ is an isomorphism $\mathcal{L}_f\to \mathcal{L}_{f^\prime}$ and we have the isomorphisms $\mathcal{L}_{f_{1}}\cong \mathcal{L}_{f_{8}}$, $\mathcal{L}_{f_{2}}\cong \mathcal{L}_{f_{4}}$, $\mathcal{L}_{f_{3}}\cong \mathcal{L}_{f_{12}}$, $\mathcal{L}_{f_{5}}\cong \mathcal{L}_{f_{10}}$,  $\mathcal{L}_{f_{7}}\cong \mathcal{L}_{f_{14}}$, and $\mathcal{L}_{f_{11}}\cong \mathcal{L}_{f_{13}}$. Any loop $\mathcal{L}_{f_{i}}$ is isomorphic to one of $\mathcal{L}_{f_{0}}$, $\mathcal{L}_{f_{1}}$, $\mathcal{L}_{f_{2}}$, $\mathcal{L}_{f_{3}}$, $\mathcal{L}_{f_{5}}$, $\mathcal{L}_{f_{6}}$, $\mathcal{L}_{f_{7}}$, $\mathcal{L}_{f_{9}}$, $\mathcal{L}_{f_{11}}$ and $\mathcal{L}_{f_{15}}$  
 given by the functions  
 \begin{equation}\label{tabla1}\begin{array}{c||c |c |c |c |c |c |c |c |c |c |}
& f_0 & f_1 & f_2 & f_3 & f_5 & f_6 & f_7 & f_9 & f_{11} & f_{15}\\
 \hline\hline
 (1,1) & 0 & 0 & 0 & 0 & 0 & 0 & 0 & 1 & 1 & 1 \\ 
 \hline
 (1,2) & 0 & 0 & 0 & 0 & 1 & 1 & 1 & 0 & 0 & 1 \\
 \hline
 (2,1) & 0 & 0 & 1 & 1 & 0 & 1 & 1 & 0 & 1 & 1 \\	
 \hline
 (2,2) & 0 & 1 & 0 & 1 & 1 & 0 & 1 & 1 & 1 & 1 
 \end{array}\;\,.\end{equation} 
 Clearly, $\mathcal{L}_{f_0}$, $\mathcal{L}_{f_1}$, $\mathcal{L}_{f_6}$, $\mathcal{L}_{f_7}$, $\mathcal{L}_{f_9}$ and $\mathcal{L}_{f_{15}}$ are commutative.\\
For $\xi\in C_3 =\{0,1,2\}$ denote $\xi'\in C_2 =\{0,1\}$ the element given by $0' = 0$, $1' = 1$ and $2' = 0$. The map $\mathcal{L}_f\to\mathcal{L}_{f^*}$,  
 $(\xi,x)\mapsto(\xi,x + \xi')$ is an isomorphism, if and only if 
  \[(\alpha + \beta, f^*(\alpha,\beta) + a + b + \alpha' + \beta') = 
  (\alpha + \beta, f(\alpha,\beta) + a + b + (\alpha + \beta)')\]
  is an identity. Working with canonical copies of $C_2$ and $C_3$ in $\mathbb Z$ we obtain equivalently  
  \[f^*(\alpha,\beta) + \alpha + \beta = 
 f(\alpha,\beta) + (\alpha + \beta \mod 3) \mod 2,\quad \alpha, \beta\in\in\{0,1,2\}.\] 
Hence we get 
 \[\begin{split}
 	&f^*(\alpha,\beta) = f(\alpha,\beta),\quad &\hbox{if}& \quad \alpha,\beta\in\{0,1\},\\
    &f^*(\alpha,\beta) + 1 = f(\alpha,\beta),\quad &\hbox{if}& \quad \alpha \beta = 2, \\
    &f^*(\alpha,\beta) =  f(\alpha,\beta) + 1, \quad &\hbox{if}& \quad\alpha = \beta = 2.\end{split}\]
i.e. the function $f^*$ satisfies $f^*(1,1) = f(1,1)$ and $f^*(\alpha,\beta) = f(\alpha,\beta) + 1$ if $\alpha,\beta\in\{1,2\}$ and $\alpha = 2$ or $\beta = 2$. It follows that we have the isomorphisms:  $\mathcal{L}_{f_0}\cong\mathcal{L}_{f_7}(\cong\mathcal{L}_{f_{14}})\cong\mathcal{L}_{f_9}$, $\mathcal{L}_{f_1}\cong\mathcal{L}_{f_6}$,   $\mathcal{L}_{f_1}(\cong\mathcal{L}_{f_8})\cong\mathcal{L}_{f_{15}}$, $\mathcal{L}_{f_2}\cong\mathcal{L}_{f_5}$,  $\mathcal{L}_{f_2}(\cong\mathcal{L}_{f_4})\cong\mathcal{L}_{f_3}$,  $\mathcal{L}_{f_3}(\cong\mathcal{L}_{f_{12}})\cong\mathcal{L}_{f_{11}}$ and we can omit $\mathcal{L}_{f_3}$, $\mathcal{L}_{f_5}$, $\mathcal{L}_{f_6}$, $\mathcal{L}_{f_7}$, $\mathcal{L}_{f_{9}}$, $\mathcal{L}_{f_{11}}$, $\mathcal{L}_{f_{15}}$ from the further investigation. $\mathcal{L}_{f_{0}}$ is isomorphic to the cyclic group $C_6$.  $\mathcal{L}_{f_1}$ and $\mathcal{L}_{f_2}$ are commutative and non-commutative proper loops.
We obtain the following 
\begin{prop} \label{ordertwo} Any Albert extension of $C_2$ by $C_3$ is isomorphic to the one of the following loops:\; cyclic group $\mathcal{L}_{f_0}\cong C_6$,\; commutative proper loop  $\mathcal{L}_{f_1}$,\; non-commutative proper loop $\mathcal{L}_{f_9}$. 
The normal subgroup $\{0\}\times C_2$ for each loop is nuclear. 
\end{prop}

\subsubsection*{Extensions of $C_3$ by $C_2$}
In the following we identify $C_3$ with the additive group of the field $\mathbb F_3$ of order $3$. Consider the multiplication table (Latin square) of a quasigroup. Its rows (columns) are permutations describing the left (right) translations. The even permutations of $\mathbb F_3$ are the maps $x\mapsto x + q$ and the odd permutations are $x\mapsto 2x + q$ with some $q\in\mathbb F_3$, since $x\mapsto 2x$ is a transposition. Cosequently any permutation of $\mathbb F_3$ can be expressed as $x\mapsto (1 + \sigma)x + q$ with $\sigma\in\{0,1\}\subset\mathbb F_3$ and $q\in\mathbb F_3$, this permutation is even for $\sigma = 0$ and odd for $\sigma = 1$. It is easy to check that the rows (columns) of the multiplication table are permutations of the same parity. For a given parity of rows (columns), the tables can only differ in the order of the rows (columns). Hence if $x\mapsto (1 + \sigma)x + q$ is the permutation given by the first row (column), then the $i$-th row (column), $i = 2,3$, is the permutation $x\mapsto (1 + \sigma)x + q + (1 + \tau)i$, where $\tau\in\{0,1\}\subset\mathbb F_3$. It follows that a quasigroup multiplication defined on $F_3$ can be expressed by  
\begin{equation}\label{quasigroup} (a,b)\mapsto q + (1 + \sigma)a + (1 + \tau)b,\quad \text{with fixed}\quad\sigma,\tau\in\{0,1\}\subset\mathbb F_3,\; q\in\mathbb F_3.
\end{equation} 
Let $(\alpha,a)(\beta,b)= (\alpha + \beta,\,\triangledown_{\alpha,\beta}(a,b))$, $\alpha,\beta\in  C_2$, $a,b\in\mathbb F_3$, be the multiplication of an Albert extension of $C_3$ by $C_2$. Then   
\[\triangledown_{0,0}(a,b) = a + b,\quad\triangledown_{0,1}(0,b) = b,\quad\triangledown_{1,0}(0,b) = \triangledown_{1,0}(b,0) = b,\quad \text{for any}\quad a,b\in\mathbb F_3.\] 
A loop of order $3$ is necessarily a group, hence $\triangledown_{1,0}(a,b) = a + b$.  
The quasigroup with multiplication $\triangledown_{0,1}(a,b)$ is isomorphic either to  $C_3$ or $\triangledown_{0,1}(a,b) = 2a + b$, hence $\triangledown_{0,1}(a,b) = (1 + \theta)a + b$ with $\theta\in\{0,1\}\subset\mathbb F_3$.
According to (\ref{quasigroup}) we can express 
 \[\triangledown_{1,1}(a,b) = q + (1 + \sigma)a + (1 + \tau)b,\quad \sigma, \tau\in\{0,1\}\subset\mathbb F_3,\; q\in\mathbb F_3.\] 
For $\xi\in C_2 =\{0,1\}$ denote $\xi'\in\mathbb F_3$ the canonical copy of elements of $\{0,1\}$ in $\mathbb F_3$, then  
\begin{equation}\label{xieta}(\xi + \eta)' = \xi' + \eta' + \xi'\eta'\quad\text{for any}\quad \xi,\eta\in C_2 =\{0,1\}.\end{equation} 
We have the decomposition       
\[\triangledown_{\alpha,\beta} = \alpha'\beta'\big(\triangledown_{1,1} - \triangledown_{1,0} - \triangledown_{0,1}\big) + \alpha'\triangledown_{1,0} + \beta'\triangledown_{0,1} + (1 - \alpha')(1 - \beta')\triangledown_{0,0}.\] 
Hence we get   
\begin{equation}\label{nablaalphabeta}\begin{split}\triangledown_{\alpha,\beta}(a,b) &= \alpha'\beta'\big(q + (1 + \sigma)a + (1 + \tau)b - a - b - (1 + \theta)a - b\big) +\\
&+ \alpha'\big(a + b\big) + \beta'\big((1 + \theta)a + b\big) + (1 - \alpha')(1 - \beta')\big(a + b\big) = \\&= \alpha'\beta'\big(q + (\sigma - \theta) a + \tau b\big) + (1 + \beta'\theta) a +  b.\end{split}\end{equation}
Denote $\epsilon = (\sigma - \theta)(1 + \theta) = \sigma + \theta + \sigma\theta$, and hence $\sigma = \epsilon + \theta + \epsilon\theta$, where  $\epsilon\in\{0,1\}\subset\mathbb F_3$. Putting this into (\ref{nablaalphabeta}) we obtain:  
\[\triangledown_{\alpha,\beta}(a,b) = \alpha'\beta'\big(q + \epsilon(1 + \theta)a + \tau b\big) + (1 + \beta'\theta)a + b,\] 
hence the multiplication can be written in the form 
\begin{equation}\label{ximu}(\alpha,a)(\beta,b) = 
(\alpha + \beta,\alpha'\beta'\big(q + \epsilon(1 +\theta)a + \tau b\big) + (1 + \beta'\theta)a + b),\end{equation}
 for any $\alpha,\beta\in C_2$, $a,b\in C_3$, where $\epsilon,\theta,\tau\in\{0,1\}$ and $q \in\{0,1,2\}$ are fixed parameters.\\ 
The map $C_2\times C_3 \to C_2\times C_3$, $(\alpha,a)\mapsto(\alpha,2a)$ is an isomorphism between the loops with parameters $\epsilon$, $\theta$, $\tau$, $q$, respectively  $\epsilon$, $\theta$, $\tau$, $q^* = 2q$, since 
\[\alpha'\beta'\big(q^* + \epsilon(1 +\theta)2a + \tau 2b\big) + (1 + \beta'\theta)2a + 2b =\]
\[ = 2\big(\alpha'\beta'\big(q + \epsilon(1 +\theta)a + \tau b\big) + (1 + \beta'\theta)a + b\big),\quad \alpha,\beta\in C_2,\; a, b\in C_3.\]
Consequently, we can assume $q\in\{0,1\}$.  \\
We obtain the assertion of 
\begin{prop} \label{cyclic}  Let $\mathcal{L}(\epsilon,\theta,\tau,q)$ be the Albert extension defined on $C_2\times C_3$ by the multiplication 
\begin{equation}
\label{Albertmult}(\alpha,a)(\beta,b) = (\alpha + \beta,\alpha'\beta'\big(q + \epsilon(1 + \theta)a + \tau b\big) + (1 + \beta'\theta)a + b),\quad\alpha,\beta\in C_2,\quad a,b\in C_3,\end{equation}
depending on the parameters $\epsilon, \theta, \tau, q\in\{0,1\}$. Any loop of order six  containing a subgroup of order three is isomorphic to an Albert extension $\mathcal{L}(\epsilon,\theta,\tau,q)$.\\
$\mathcal{L}(\epsilon,\theta,\tau,q)$ is commutative if and only if $\epsilon = \tau$ and $\theta = 0$. \\
An Albert extension $\mathcal{L}(\epsilon,\theta,\tau,q)$  is either nuclear or it has precisely one of the left, middle or right nuclear property.  $\mathcal{L}(\epsilon,\theta,\tau,q)$ is left, middle or right nuclear if and only if $\epsilon = 0$, $\epsilon = \tau$ or $\tau = 0$, respectively. \\  
$\mathcal{L}(\epsilon,\theta,\tau,q)$ is a  group if and only if $\epsilon = \tau =0$ and $\theta q = 0$.
\end{prop}
\begin{proof}  First we prove that the multiplication (\ref{Albertmult}) determines a loop. Clearly, $(0,0)$ is the unit element of $\mathcal{L}(\epsilon,\theta,\tau,q)$. The equation $(\xi,x)(\beta,b) = (\gamma,c)$ for $(\xi,x)$ yields $\eta = \beta + \gamma$ and the linear equation 
\[(\beta + \gamma)'\beta'\big(q + \epsilon(1 + \theta)x + \tau b\big) + (1 + \beta'\theta)x + b\big) = c\]	
for $x$. The coefficient of $x$ is 
\begin{equation*}
(\beta + \gamma)'\beta'\epsilon(1 + \theta) + (1 + \beta'\theta) = \begin{cases}
1,              & \text{if}\;\beta = 0\\
\big((1 + \gamma)'\epsilon + 1\big)(1 + \theta),              & \text{if}\;\beta = 1.
\end{cases}
\end{equation*}
Since $1 + \gamma\in C_2$ the coefficient of $x$ is contained in $\{1,2\}$, hence the equation is uniquely solvable.
Similarly the equation $(\alpha,a)(\eta,y) = (\gamma,c)$ for $(\eta,y)$ gives $\eta = \alpha + \gamma$ and \[\alpha'(\alpha + \gamma)'\big(q + \epsilon(1 + \theta)a + \tau y\big) + (1 + \beta'\theta)a + y = c,\] 
which is uniquely solvable, since the coefficient of $y$ is $\alpha'(\alpha + \gamma)'\tau + 1\in\{1,2\}$. Clearly, for any loop of order $6$ a subgroup of order $3$ is necessarily normal, hence $\mathcal{L}(\epsilon,\theta,\tau,q)$ is a loop which is an Albert extension.\\ 
The associativity of the  multiplication is equivalent to the identity  
\begin{equation}\label{assoc}\begin{split} &(\alpha' + \beta' + \alpha'\beta')\gamma'\Big(q + \epsilon(1 + \theta)[\alpha'\beta'(q + \epsilon(1 + \theta)a +\tau b) + (1 +\beta'\theta)a + b] + \tau c\Big) +\\
&+ (1 + \gamma'\theta)\Big(\alpha'\beta'(q + \epsilon(1 + \theta)a +\tau b) + (1 +\beta'\theta)a + b\Big) + c = \\ &= \alpha'(\beta' +\gamma' + \beta'\gamma')\Big(q + \epsilon(1 + \theta)a + \tau [\beta'\gamma'(q + \epsilon(1 + \theta)b + \tau c) +  (1 + \gamma'\theta)b + c]\Big) + \\ &+ \Big(1 + (\beta' +\gamma' + \beta'\gamma')\theta\Big)a + \beta'\gamma'\Big(q + \epsilon(1 + \theta)b + \tau c\Big) +  (1 + \gamma'\theta)b + c, \end{split} \end{equation} 
where $\alpha,\beta,\gamma\in C_2$,\; $a,b,c\in C_3$.
Substituting one by one $\alpha = 0$, $\beta = 0$ and $\gamma = 0$ we obtain that   $\mathcal{L}(\epsilon,\theta,\tau,q)$ is a left, middle or right nuclear extension if and only if $\epsilon = 0$, $\epsilon = \tau$ or $\tau = 0$, respectively.  It follows that $\mathcal{L}(\epsilon,\theta,\tau,q)$ is nuclear if and only if $\epsilon = \tau = 0$. 
Clearly, $\mathcal{L}(\epsilon,\theta,\tau,q)$ is commutative if and only if $\epsilon = \tau$ (middle nuclear) and $\theta = 0$. 
Putting $\epsilon = \tau = 0$ into (\ref{assoc}) we get that $\mathcal{L}(\epsilon,\theta,\tau,q)$ is a group if and only if $\epsilon = \tau = 0$ (nuclear) and  $\theta q = 0$. 
\end{proof}
The map $\mathcal{L}(\epsilon,\theta,\tau,q)\to \mathcal{L}(\epsilon,\theta,\tau,q^*)$,   $(\alpha,a)\mapsto(\alpha,\alpha' + a)$ is an isomorphism if and only if  \[\alpha'\beta'\Big(q^* + \epsilon(1 + \theta)\alpha' + \tau \beta'\Big) +  \alpha'\beta'\theta =  \alpha'\beta'(q + 1)\quad\text{for any}\quad\alpha,\beta\in C_2,\] 
or equivalently  $q^* + \epsilon(1 + \theta) + \theta + \tau = q + 1$. Computing 
\[\epsilon(1 + \theta) + \theta + \tau =\begin{cases}0&\text{if}\quad(\epsilon,\theta,\tau) = (0,0,0),\; (1,1,0),\\2&\text{if}\quad(\epsilon,\theta,\tau) = (0,1,1),\; (1,0,1),\end{cases}\]
we obtain the isomorphisms $\mathcal{L}(0,0,0,0)\cong\mathcal{L}(0,0,0,1)$ and $\mathcal{L}(1,1,0,0)\cong\mathcal{L}(1,1,0,1)$, respectively  
$\mathcal{L}(0,1,1,0)\cong\mathcal{L}(0,1,1,1)$ and $\mathcal{L}(1,0,1,0)\cong\mathcal{L}(1,0,1,1)$. It follows that the loops $\mathcal{L}(0,0,0,1)$, $\mathcal{L}(0,1,1,1)$, $\mathcal{L}(1,0,1,1)$ and  $\mathcal{L}(1,1,0,1)$ can be omitted from the further investigation. \\
The squares of elements belonging to $\{1\}\times C_3$ are 
\[(1,a)^2  = (0,q + \big((1 + \epsilon)(1 + \theta) + 1 + \tau\big)a) =\begin{cases} = (0,q + \tau a),&\text{if}\quad  \epsilon \neq \theta, \\ = (0,q + (2 + \tau)a),&\text{if}\quad \epsilon =\theta.\end{cases}\]
Hence the set $\{(1,a)^2,\; a\in C_3\}$ is either $\{(0,q)\}$ with $q\in\{0,1\}$ or $\{0\}\times C_3$. The extensions, their nuclear properties and the corresponding sets $\{(1,a)^2,\; a\in C_3\}$ are given by the list:
\begin{equation}\label{list}\begin{array}{|c |l |l |}
\hline
\textbf{extension} & \textbf{associativity property} &\bf{\{(1,a)^2,\; a\in C_3\}} \\ 
\hline
\mathcal{L}(0,0,0,0) & \text{cyclic group} & \{0\}\times C_3 \\ 
\hline
\mathcal{L}(0,0,1,0) & \text{dihedral group} & \{(0,0)\} \\
\hline
\mathcal{L}(0,0,1,1) & \text{non-associative nuclear} & \{(0,1)\} \\	
\hline
\mathcal{L}(0,1,0,0)  & \text{left nuclear} & \{(0,0)\} \\
\hline
\mathcal{L}(0,1,0,1)  & \text{left nuclear} & \{(0,1)\}  \\
\hline
\mathcal{L}(0,1,1,0)  & \text{left nuclear} & \{0\}\times C_3  \\
\hline
\mathcal{L}(1,0,0,0) & \text{right nuclear} & \{(0,0)\} \\ 
\hline
\mathcal{L}(1,0,0,1) & \text{right nuclear} & \{(0,1)\} \\ 
\hline
\mathcal{L}(1,0,1,0) & \text{right nuclear} & \{0\}\times C_3 \\ 
\hline
\mathcal{L}(1,1,0,0) & \text{commutative middle nuclear} & \{0\}\times C_3\\ 
\hline
\mathcal{L}(1,1,1,0) & \text{middle nuclear} & \{(0,0)\} \\ 
\hline
\mathcal{L}(1,1,1,1) & \text{middle nuclear} & \{(0,1)\} \\ 
\hline
\end{array}\end{equation} 
These loops have different algebraic properties, hence they are not isomorphic. It follows    
\begin{prop} \label{6loops} Any Albert extension of $C_2$ by $C_3$ is isomorphic to one of the non-isomorphic loops listed in (\ref{list}).  $\mathcal{L}(0,0,0,0)$ is isomorphic to the cyclic group $C_6$, $\mathcal{L}(0,0,1,0)$  is isomorphic to the dihedral group $D_3$ and $\mathcal{L}(0,0,1,1)$ is a non-associative nuclear extension. The others are proper loops with non-nuclear $\{0\}\times C_3$. Only the middle nuclear $\mathcal{L}(1,1,0,0)$ is non-associative and commutative. 
\end{prop}

\subsection{Right nuclear extensions}

In the following we discuss nuclear properties of normal subloops of loops. Since the left, right or middle nuclei of a loop are necessarily groups, we will investigate extensions of groups $G$ by loops $K$. In the following $\mathrm{Sym}(M)$ denotes the group of all permutations of a set M.
\begin{lem} \label{Lgamma} An Albert extension $K\times^\triangledown G$ of $G$ by $K$ is right nuclear if and only if the function $\triangledown_{\xi,\eta}(x,y)$ has the form $\triangledown_{\xi,\eta}(x,y) = \Gamma_{\xi,\eta}(x)y$, $\xi,\eta\in K$, $x,y\in G$, where $\Gamma: K\times K\to\mathrm{Sym}(G)$ is a map with $\Gamma_{\epsilon,\sigma}(e) = e$ and $\Gamma_{\sigma,\epsilon} = \mathrm{Id}$ for all $\sigma\in K$.
\end{lem}
\begin{proof} If $K\times^\triangledown G$ is right nuclear if and only if 
$\triangledown_{\alpha\beta,\epsilon}(\triangledown_{\alpha,\beta}(a,b),c) = \triangledown_{\alpha,\beta}(a,\triangledown_{\beta,\epsilon}(b,c))$ is an identity. Putting $a = e$, $\beta = \epsilon$ we get $\triangledown_{\alpha,\epsilon}(b,c) = bc$ for any $\alpha\in K$, $b,c\in G$. Replacing this into the right nuclear identity we get $\triangledown_{\alpha,\beta}(a,b)c = \triangledown_{\alpha,\beta}(a,bc)$, i.e. $\triangledown_{\alpha,\beta}(a,c) = \triangledown_{\alpha,\beta}(a,e)c$ for all $\alpha,\beta\in K$, $a,c\in G$. Denoting $\Gamma_{\xi,\eta}(x) = \triangledown_{\xi,\eta}(x,e)$ we obtain the first implication. Conversely, if $\triangledown_{\xi,\eta}(x,y) = \Gamma_{\xi,\eta}(x)y$, then we have the equation 
\[\triangledown_{\alpha\beta,\epsilon}(\triangledown_{\alpha,\beta}(a,b),c) = \Gamma_{\alpha,\beta}(a)bc = \triangledown_{\alpha,\beta}(a,\triangledown_{\beta,\epsilon}(b,c)),\]
hence $K\times^\triangledown G$ is right nuclear. 
\end{proof}
\begin{defi} \label{Rinuex} Let $\Gamma: K\times K\to\mathrm{Sym}(G)$ be a map satisfying $\Gamma_{\epsilon,\sigma}(e) = e$ and $\Gamma_{\sigma,\epsilon} = \mathrm{Id}$ for all $\sigma\in K$. We denote by $K\times^\Gamma G$ the \emph{right nuclear Albert extension} of $G$ by $K$ defined by the multiplication $(\alpha,a)(\beta,b)= (\alpha\beta,\Gamma_{\alpha,\beta}(a)b)$, $\alpha,\beta\in K$, $a,b\in G$. 
\end{defi}
\begin{lem}  \label{Psialphabeta}  Assume that the loop $L$ is a right nuclear extension of the normal proper subgroup  $G\lhd L$  by the loop $K$, $\kappa: K\to L/G$ is the underlying isomorphism and $\Sigma$ is a left transversal to $G$ in $L$. Let  $\Gamma: K\times K\to \mathrm{Sym}(G)$ be the map  
defined by 	
\begin{equation} \label{Gammaxieta} \Gamma_{\xi,\eta}(x) = l_{\xi\eta}\backslash (l_\xi x \cdot l_\eta),\quad \xi,\eta\in K, \; x\in G, \quad\text{where} \quad \{l_\sigma\} = \Sigma\cap\kappa(\sigma),\quad \sigma\in K\end{equation} 
and let $K\times^\Gamma G$ be the Albert extension determined by the  multiplications $\triangledown_{\alpha,\beta}(a,b) = \Gamma_{\alpha,\beta}(a)b$. Then the map  $\Phi:(\xi,x) \mapsto l_\xi x$ is an isomorphism $\Phi:K\times^\Gamma G\to L$.
\end{lem}
\begin{proof}  Applying the natural homomorphism $L\to L/G$ to $l_{\xi\eta}\backslash (l_\xi x \cdot l_\eta)$ we get that $\Gamma_{\xi,\eta}(x)\in G$ for any $\xi,\eta\in K$, $ x\in G$, hence $\Gamma_{\xi,\eta}\in\mathrm{Sym}(G)$. The bijective map $\Phi:K\times^\Gamma G\to L$ is an isomorphism since  
$\Phi((\xi,x)(\eta,y)) = \Phi(\xi,x)\Phi(\eta,y)$ follows from 
\[l_{\xi\eta}\big(l_{\xi\eta}\backslash (l_\xi x \cdot l_\eta)\cdot y\big) = \big(l_{\xi\eta}\cdot l_{\xi\eta}\backslash (l_\xi x \cdot l_\eta)\big) y = 
(l_\xi x\cdot l_\eta) y = l_\xi x\cdot l_\eta y.\]
Hence we obtain the assertion.
\end{proof}  
\begin{rem}
If we denote the map $\Gamma_{\epsilon,\tau}:K\to\mathrm{Sym}(G)$ and the function $\Gamma(e):K\times K\to G$ by $\Theta_\tau = \Gamma_{\epsilon,\tau}$ and $f(\sigma,\tau) = \Gamma_{\sigma,\tau}(e)$ then we get the expressions 
\[\Theta_\tau = \Gamma_{\epsilon,\tau}(x) = l_{\eta}\backslash x l_\eta\quad \text{and}\quad f(\sigma,\tau) = \Gamma_{\sigma,\tau}(e) = l_{\xi\eta}\backslash l_\xi l_\eta.\]
\end{rem}
\begin{thm} \label{middle} The right nuclear Albert extension $K\times^\Gamma G$ is 
	\begin{enumerate} 
		\item [\emph{(i)}] middle nuclear if and only if all $\Theta_\tau$, $\tau\in K$, are automorphisms of $G$ and $\Gamma$ has the form $\Gamma_{\sigma,\tau}(s) = f(\sigma,\tau)\Theta_\tau(s)$,   
		\item [\emph{(ii)}] left nuclear if and only if  all $\Theta_\tau$, $\tau\in K$, are automorphisms of $G$ and $\Gamma$ has the form $\Gamma_{\sigma,\tau}(g) = \Theta_{\sigma\tau}(\Theta^{-1}_\sigma(g)) f(\sigma,\tau)$,
	  \item [\emph{(iii)}] nuclear if and only if $\Gamma_{\sigma,\tau}(s) = f(\sigma,\tau)\Theta_\tau(s)$ and $\Theta_{\sigma\tau}\Theta_\sigma^{-1}\Theta_\tau^{-1} = \iota_{f(\sigma,\tau)}$, $\sigma,\tau\in K$, where $\iota_s(t) = sts^{-1}$, $s,t\in G$.
	\end{enumerate}
\end{thm}
\begin{proof} The multiplication of $K\times^\Gamma G$ is associative if and only if $K$ is a group and 
	\begin{equation}\label{asslnuc}\Gamma_{\alpha\beta,\gamma}(\Gamma_{\alpha,\beta}(a)b) =  \Gamma_{\alpha,\beta\gamma}(a)\Gamma_{\beta,\gamma}(b)\end{equation} 
for any $\alpha,\beta,\gamma\in K$ and $a,b\in G$. 
Putting $\beta = \epsilon$ we get that the middle nuclear property of $K\times^\Gamma G$ is equivalent to the identity
$\Gamma_{\alpha,\gamma}(ab) = \Gamma_{\alpha,\gamma}(a)\Theta_\gamma(b)$,  
and with $a = e$ we get the necessary condition  $\Gamma_{\alpha,\gamma}(b) = f(\alpha,\gamma)\Theta_\gamma(b)$. Hence $K\times^\Gamma G$ is middle nuclear if and only if 
 \[\Gamma_{\alpha,\gamma}(ab) = f(\alpha,\gamma)\Theta_\gamma(ab) = \Gamma_{\alpha,\gamma}(a)\Theta_\gamma(b) = f(\alpha,\gamma)\Theta_\gamma(a)\Theta_\gamma(b),\]
consequently $\Theta_\gamma\in\mathrm{Aut}(G)$ for any $\gamma\in K$ giving the assertion (i).\\
Replacing $\alpha = \epsilon$ into the identity (\ref{asslnuc}) we get 
\begin{equation}\label{lenupr}\Gamma_{\beta,\gamma}(\Theta_\beta(a)b) =  \Theta_{\beta\gamma}(a)\Gamma_{\beta,\gamma}(b).\end{equation} 
$\beta = \epsilon$ yields $\Theta_\gamma(ab) =  \Theta_\gamma(a)\Theta_\gamma(b)$, i.e. $\Theta_\gamma\in\mathrm{Aut}(G)$ for any $\gamma\in K$.
Putting $b = e$ and $c = \Theta_\beta(a)$ we get 
\begin{equation}\label{leftnuc} \Gamma_{\beta,\gamma}(c) =   \Theta_{\beta\gamma}(\Theta^{-1}_\beta(c))f(\beta,\gamma).\end{equation} 
Conversely, if (\ref{leftnuc}) holds for any $\beta,\gamma\in K$, $c\in G$ then 
the (\ref{lenupr}) is satisfied and we obtain the assertion (ii). \\
Assume now, that both (i) and (ii) are satisfied, i.e.  \[\Gamma_{\beta,\gamma}(b) = f(\beta,\gamma)\Theta_\gamma(b) = \Theta_{\beta\gamma}(\Theta^{-1}_\beta(b))f(\beta,\gamma).\] 
 Putting $a = \Theta_\gamma(b)$ we get 
$f(\beta,\gamma)af(\beta,\gamma)^{-1} = \Theta_{\beta\gamma}\circ\Theta^{-1}_\beta\circ\Theta^{-1}_\gamma(a)$.
Conversely, using this equation we get $f(\beta,\gamma)\Theta_\gamma(b) = \Theta_{\beta\gamma}(\Theta^{-1}_\beta(b))f(\beta,\gamma)$, hence the assertion (iii) follows. 
\end{proof}
\begin{cor} If $G$ is an abelian group then $K\times^\Gamma G$ is a nuclear extension of $G$ by $K$ if and only if $\Gamma_{\sigma,\tau}(s) = f(\sigma,\tau)\Theta_\tau(s)$ and $\Theta:K\to\mathrm{Aut}(G)$ is an anti-homomorphism. 
\end{cor}
\begin{cor} (A. Dr\'apal and P. Vojt\v{e}chovsk\'y, \cite{DrV} Theorem 2.2.) Let $G$ be an abelian group and $K$, $L$ loops. Then the following conditions are equivalent:
	\begin{enumerate}
		\item $L$ is a nuclear extension of $G$ by $K$,
		\item $L$ is isomorphic to the Bruck extension $K\times^\triangledown G$  with $\triangledown_{\sigma,\tau}(s,t) = s\varphi_\sigma(t)f(\sigma,\tau)$,  
where $f:K\times K\to G$ fulfils $f(\epsilon,\tau) = f(\sigma,\epsilon) = e$ and $\varphi:K\to \mathrm{Aut}(G)$ is a homomorphism.
	\end{enumerate} 
\end{cor}
 \begin{proof} According to Theorem \ref{brualb} $K\times^\triangledown G$ is isomorphic to the Albert extension $K\times^{\triangledown^*}N$ determined by  
$\triangledown^*_{\sigma,\tau}(s,t) = \varphi_{\sigma\tau}^{-1}( \varphi_\sigma(s)\,\varphi_\sigma(\varphi_\tau(t))\,f(\sigma,\tau)) = f(\sigma,\tau)\,\varphi_{\tau}^{-1}(s)\,b$, 
 	where $\varphi_\sigma$ is the map $x\mapsto \triangledown_{\sigma,\epsilon}(e,x)$, $\sigma\in K$. Putting $\Theta = \varphi^{-1}$ we get the assertion.
 \end{proof}

 \subsection{Right nuclear extensions with left or right inverse property}
 
Let $K\times^\Gamma G$ be a right nuclear Albert extension of the group $G$ by the group $K$. In the following we investigate the fulfilment of the left, respectively right inverse property in $K\times^\Gamma G$ with the additional condition that the left, respectively right inverses are of the form $(\alpha^{-1},a^{-1})$.\\ 
Let $\pi:K\times K\to K\times K$ be the involutory permutation $\pi(\alpha,\beta) = (\alpha^{-1},\alpha\beta)$. Define the action of $\pi$ on $\mathrm{Sym}(G)$ by $\varphi^\pi(g) = \varphi(g^{-1})^{-1}$, $\varphi\in\mathrm{Sym}(G)$, $g\in G$. 
 \begin{lem}\label{lip} A right nuclear Albert extension $K\times^\Gamma G$ has the left inverse property with  left inverses $(\alpha^{-1},a^{-1})$, $\alpha\in K$, $a\in G$, if and only if $\Gamma$ is equivariant under the action of $\pi$.\end{lem}
\begin{proof}
$(\alpha^{-1},a^{-1})$ is the left inverse of $(\alpha,a)$ if and only if  $\Gamma_{\alpha^{-1},\alpha\beta}(a^{-1})\Gamma_{\alpha,\beta}(a) = e$, or equivalently 
 \begin{equation}\label{invol} \Gamma_{\pi(\alpha,\beta)}(g) = \Gamma_{\alpha^{-1},\alpha\beta}(g) = \Gamma_{\alpha,\beta}(g^{-1})^{-1} = \Gamma^\pi_{\alpha,\beta}(g)\end{equation}
holds for any $\alpha,\beta\in K$,  $g\in G$.
\end{proof}
Let $\Pi$ denote the group  generated by $\pi$ and denote $\Xi = \{(\sigma,\epsilon),\; \sigma\in K\}\cup \{(\sigma^{-1},\sigma),\; \sigma\in K\}$. Clearly, the set $\Xi$ consists of orbits  $\{(\sigma^{-1},\sigma),(\sigma,\epsilon)\},\; \sigma\in K\}$ of $\Pi$. 
The following construction yields a large class of right nuclear Albert extensions having the left inverse property.
\begin{ex} \label{equivar} Define the map $\Gamma:K\times K\to\mathrm{Sym}(G)$ by  $\Gamma_{\alpha^{-1},\alpha} = \Gamma_{\alpha,\epsilon} = \mathrm{Id}$ on $\Xi$. Choose the value $\Gamma_{\sigma,\tau}\in\mathrm{Sym}(G)$ for an element of any orbit of the group $\Pi$ in $K\times K\setminus \Xi$ arbitrarily and determine the map $\Gamma$ assuming that the map $\Gamma$ is equivariant under $\Pi$. The obtained right nuclear Albert extensions $K\times^\Gamma G$ are neither middle nor left nuclear, if some values of $\Gamma$ are not belonging to the set 
	\[\{\lambda_g\Theta,\; g\in G, \Theta\in\mathrm{Aut}(G)\}\cup\{\rho_g\Theta,\; g\in G, \Theta\in\mathrm{Aut}(G)\}.\]
\end{ex}
These examples show that there are many right nuclear Albert extensions having the left inverse property which are neither middle nor left nuclear. In contrast to the left inverse property the right inverse property is rather restrictive for right nuclear Albert extensions. \\
Let $\rho:K\times K\to K\times K$ be the involutory permutation $\rho(\alpha,\beta) = (\alpha\beta,\beta^{-1})$. Define the action of $\rho$ on $\mathrm{Sym}(G)$ by $\varphi^\rho = \varphi^{-1}$, $\varphi\in\mathrm{Sym}(G)$.
\begin{lem} A right nuclear Albert extension $K\times^\Gamma G$  has the right inverse property with right inverses $(\alpha^{-1},a^{-1})$, $\alpha\in K$, $a\in G$,  if and only if $\Gamma_{\alpha,\beta}(g) = f(\alpha,\beta)g$ and $\Gamma$ is equivariant under the action of $\rho$.
\end{lem} 
\begin{proof} $(\beta^{-1},b^{-1})\in K$ is the right inverse of $(\beta,b)$ if and only if for any $(\alpha,a)$ we have $\Gamma_{\alpha\beta,\beta^{-1}}(\Gamma_{\alpha,\beta}(a)b) = ab$, or equivalently $\Gamma_{\alpha,\beta}(a)b = \Gamma_{\alpha\beta,\beta^{-1}}^{-1}(ab)$. Hence  
\begin{equation}\label{rightinv0} \Gamma^\rho_{\alpha,\beta} = \Gamma_{\alpha\beta,\beta^{-1}} = \Gamma_{\rho(\alpha,\beta)},
\end{equation} 
 consequently $\Gamma_{\alpha,\beta}(b) = \Gamma_{\alpha,\beta}(e b) = \Gamma_{\alpha\beta,\beta^{-1}}^{-1}(eb) = \Gamma_{\alpha,\beta}(e)b = f(\alpha,\beta)b$.  
Moreover the condition (\ref{rightinv0}) yields $f(\alpha\beta,\beta^{-1}) = f(\alpha,\beta)^{-1}$. Conversely,  $\Gamma_{\alpha,\beta}(b) = f(\alpha,\beta)b$ satisfies (\ref{rightinv0}) if  $f(\alpha\beta,\beta^{-1}) = f(\alpha,\beta)^{-1}$, which proves the assertion.
\end{proof}
\begin{cor} A right nuclear Albert extension $K\times^\Gamma G$ having the right inverse property with right inverses $(\alpha^{-1},a^{-1})$, $\alpha\in K$, $a\in G$,  is middle nuclear, too.
\end{cor}

\section{Schreier extensions}


In the following we will investigate right and middle nuclear Albert extensions determined by a map $\Gamma: K\times K\to \mathrm{Sym}(G)$. According to Theorem \ref{middle} (i) in this case the map $\Gamma$ has the form $\Gamma_{\sigma,\tau}(s) = f(\sigma,\tau)\Theta_\tau(s)$, where $\Theta_\tau\in\mathrm{Aut}(G)$ for any $\tau\in K$. 
\begin{defi}  A right and middle nuclear Albert extension is called \emph{Schreier extension}.  A Schreier extension determined by the map $\Gamma_{\sigma,\tau}(s) = f(\sigma,\tau)\Theta_\tau(s)$ will be denoted by $K\times_f^\Theta G$. 
\end{defi} 
We notice that the multiplication of Schreier extensions have the same expression as in Schreier theory of group extensions. The smallest Schreier extensions are described in  Proposition \ref{ordertwo} and Proposition \ref{6loops}.
\begin{rem} The non-associative Schreier extensions of order $6$ of a non-trivial  subgroup are the extensions of a subgroup of order $2$ and a unique extension of a subgroup of order $3$ isomorphic to the loop $\mathcal{L}(0,0,1,1)$.
\end{rem}
It is well known from the classical theory of group extensions that a Schreier extension 
$K\times_f^\Theta G$ is a group if and only if $K$ is associative and the relations 
\begin{equation}\label{lnuc}\Theta_{\sigma\tau}\Theta_\sigma^{-1}\Theta_\tau^{-1} = \iota_{f(\sigma,\tau)},\quad \sigma,\tau\in K,\quad\text{where}\quad \iota_s(t) = sts^{-1},\, s,t\in G.\end{equation}  
as well as the identities 
\begin{equation}\label{faktor}f(\sigma,\tau\rho)^{-1}f(\sigma\tau,\rho)\Theta_\rho\left(f(\sigma,\tau)\right)f(\tau,\rho)^{-1} = e,\quad 
	\sigma,\tau,\rho\in K\end{equation} 
are satisfied, (cf. \cite{Kur}, \S 48). \\
Example \ref{equivar} shows that there are many right nuclear, but  neither middle nor left nuclear Albert extensions $K\times^\Gamma G$ with left inverse property. In contrast to this, we have  
\begin{lem}\label{NS} 
Any Schreier extension $K\times_f^\Theta G$ having the left inverse, the left alternative or the flexible property is nuclear.
\end{lem}
\begin{proof}
It follows from \cite{NS} Propositions 3.7, 3.8, 3.10 that any Schreier extension $K\times_f^\Theta G$ having the left inverse property satisfies  the identity (\ref{lnuc}), hence $K\times_f^\Theta G$ is also left nuclear, (c.f. Theorem \ref{middle} (iii)).  
\end{proof}
\begin{defi} A Schreier extension $K\times_f^\Theta G$ is called  \emph{automorphism-free} if $\Theta_\sigma = \mathrm{Id}$ for all $\sigma\in K$, respectively \emph{factor-free} if $f(\sigma,\tau) = e$ for all $\sigma, \tau\in K$. 
\end{defi}
\begin{cor} An automorphism-free Schreier extension $K\times_f^\mathrm{Id} G$ is nuclear if and only if all values of the function $f$ are contained in the centre of $G$.
\end{cor} 
These corollaries allow to construct examples of Schreier extensions which are  not left nuclear. 

\subsection{Non-nuclear Schreier extensions satisfying the right Bol identity}

According to Lemma \ref{NS} a Schreier extensions $K\times_f^\Theta G$ satisfying some weak-associativity properties are nuclear. Since the left Bol property is much stronger, this observation is valid for Schreier extensions with the left Bol identity. In contrast to this, it is possible to provide large classes of examples of right Bol non-nuclear Schreier extensions. We will use an  extension method of Bol loops, called Scheerer extension,  developed for   classification of compact (left) Bol loops in \cite{nsl}, pp. 42-52. \\
Let $\mathcal{G}_r(K)$ be the group generated by the set $\mathcal{P} = \{\rho_\sigma,\; \sigma\in K\}$ of right translations of the loop $K$. The stabilizer of the identity element $\epsilon\in K$ in $\mathcal{G}_r(K)$  is the right inner mapping group $\mathcal{H}_r(K)$ of $K$ generated by the maps $\rho^{-1}_{\tau\sigma} \rho_\sigma \rho_\tau: K\to K$, $\sigma, \tau\in K$, (cf. the analogous assertion for the left inner mapping group in \cite{nsl}, Lemma 1.31). The set $\mathcal{P}$ of right translations is a right transversal to $\mathcal{H}_r(K)$ in $\mathcal{G}_r(K)$, we denote by $\pi:\mathcal{G}_r(K)\to\mathcal{P}$ the map assigning to $x\in \mathcal{G}_r(K)$ the unique element of $\mathcal{P}$ contained in the right coset $\mathcal{H}_r(K)x$. The multiplication on $\mathcal{P}$ defined by $(x,y)\mapsto \pi(xy):\mathcal{P}\times\mathcal{P}\to\mathcal{P}$ defines a loop  
isomorphic to $K$ (cf. \cite{nsl}, p. 18). 
\begin{prop} Let $K$ be a right Bol loop, $\chi:\mathcal{H}_r(K)\to G$ a homomorphism of the right inner mapping group to the group $G$. The Schreier extension $K\times_f^{\mathrm{Id}} G$ determined by 
	\begin{equation}\label{HrK} f(\tau ,\sigma) = \chi(\rho^{-1}_{\tau\sigma} \rho_\sigma \rho_\tau),\; \tau, \sigma\in K, \quad\quad \Theta_\sigma=\mathrm{Id}, \; \sigma\in K \end{equation}
has the following properties:
	\begin{enumerate}
		\item[(\emph{i)}] it is a right Bol loop,  
		\item[(\emph{ii)}] it is left nuclear if and only if the image $\chi(\mathcal{H}_r(K))$ of $\mathcal{H}_r(K)$  is contained in the centre of $G$, 
		\item[(\emph{iii)}] it is a group if and only if $K$ is a group.
	\end{enumerate}
\end{prop}
\begin{proof} The loop $K$ is a homomorphic image of $K\times_f^{\mathrm{Id}} G$, hence if $K$ is a proper loop then $K\times_f^{\mathrm{Id}} G$ is a proper loop, too. If $K$ is a group then $f(\tau ,\sigma) = e$ for any $\tau, \sigma\in K$ and 
	$K\times_f^{\mathrm{Id}} G$ is the direct product of groups, giving the assertion (iii). The elements $(\epsilon,a)$, $(\beta,b)$ and $(\gamma,c)$ associate in $K\times_f^{\mathrm{Id}} G$ if and only if   
$\chi(\rho^{-1}_{\beta\gamma}\rho_\gamma \rho_\beta)a = a\chi(\rho^{-1}_{\beta\gamma} \rho_\gamma \rho_\beta)$ for all $a\in G$ and $\beta, \gamma\in K$, hence the assertion (ii) follows. The right Bol property of $K\times_f^{\mathrm{Id}} G$ is equivalent to the identity  
	\[\chi(\rho^{-1}_{\gamma(\alpha\beta\cdot\alpha)} \rho_{\alpha\beta\cdot\alpha} \rho_\gamma)c\chi(\rho^{-1}_{\alpha\beta\cdot\alpha} \rho_\alpha\rho_{\alpha\beta}\rho^{-1}_{\alpha\beta} \rho_\beta \rho_\alpha) = \chi(\rho^{-1}_{\left(\gamma\alpha\cdot\beta\right)\alpha} \rho_{\alpha} \rho_{\beta} \rho_\alpha \rho_\gamma)c.\]
Hence the right Bol identity $\rho_{\alpha\beta\cdot\alpha}  = \rho_\alpha \rho_\beta \rho_\alpha$ in $K$ gives the assertion (1).
\end{proof}
\begin{ex} Let $K$ be a right Bol loop with non-abelian right inner mapping group $\mathcal{H}_r(K)$ and $G = \mathcal{H}_r(K)$, $\chi = \mathrm{Id}$. The Schreier extension $K\times_f^{\mathrm{Id}} G$ defined by (\ref{HrK}) is not nuclear and satisfies the right Bol loop identity.
\end{ex}
\begin{prop} \label{theta} Let $K$ be a right Bol loop, $\chi: \mathcal{G}_r(K)\to G$ a homomorphism of the group generated by right translations of $K$ to the group $G$.   The Schreier extension $K\times_f^{\mathrm{Id}} G$ determined 
\begin{equation}\label{GrK}f(\tau ,\sigma) = e,\quad \tau, \sigma\in K, \quad \Theta_\sigma=\iota_{\chi(\rho_\sigma)}, \quad \sigma\in K, \quad \iota_g(h) = ghg^{-1}, \quad g,h\in \mathcal{G}\end{equation} 
has the following properties:
\begin{enumerate}
\item[(\emph{i)}]  it is a right Bol loop,  
\item[(\emph{ii)}] it is left nuclear if and only if the image $\chi(\mathcal{H}_r(K))$ of  $\mathcal{H}_r(K)\subset \mathcal{G}_r(K)$  is contained in the centre of $G$,
\item[(\emph{iii)}] it is a group if and only if $K$ is a group.
\end{enumerate}
\end{prop}
\begin{proof} $K$ is a homomorphic image of $K\times_e^{\Theta}G$, hence if $K$ is a proper loop then $K\times_e^{\Theta}G$ is also a proper loop. If $K$ is a group then $K\times_e^{\Theta}G$ is a semi-direct product of groups and we get the assertion (iii). The elements $(\epsilon,a)$, $(\beta,b)$ and  $(\gamma,c)$ associate if and only if 
	\[\chi(\rho_\gamma)\chi(\rho_\beta) a\chi(\rho_\beta)^{-1} = \chi(\rho_{\beta \gamma}) a\chi(\rho_{\beta\gamma})^{-1}\chi(\rho_\gamma) \quad\text{for any}\quad a\in K,\; \beta,\gamma\in G.\]
Hence $K\times_e^{\Theta}G$  is left nuclear if and only if for any $\beta, \gamma\in K$ the generator  $\chi(\rho^{-1}_{\beta\gamma} \rho_\gamma \rho_\beta)$ of $\chi(\mathcal{H}_r)$ commutes with all elements of $G$, giving the assertion (ii). Since $K$ satisfies the right Bol identity $\rho_{\alpha\beta\cdot\alpha} = \rho_\alpha\rho_\beta\rho_\alpha$ the extension  $K\times_f^{\mathrm{Id}} G$ is also right Bol because the identity    
		 	\[\chi(\rho_{\alpha\beta\cdot\alpha})\;c\;\chi(\rho_{\alpha\beta\cdot\alpha})^{-1}\chi(\rho_\alpha)\chi(\rho_{\beta}) = \chi(\rho_\alpha)\chi(\rho_\beta)\chi(\rho_\alpha)\;c\;
	\chi(\rho_\alpha)^{-1}\]
holds and we obtain the assertion (i).
\end{proof}
 According to Section 1.2 in \cite{nsl} the stabilizer $\mathcal{H}_r(K)$ of the identity element $e\in K$ in the group $\mathcal{G}_r(K)$ does not contain any non-trivial normal subgroup of $\mathcal{G}_r(K)$, hence $\mathcal{H}_r(K)$ is not central in $\mathcal{G}_r(K)$. It follows from Proposition \ref{theta} (ii) that we have the following 
\begin{ex}  Let $K$ be a proper right Bol loop and let be $G = \mathcal{G}_r(K)$, $\chi = \mathrm{Id}$. The Schreier extension $K\times_e^{\Theta} G$ defined by (\ref{GrK}) is not nuclear and satisfies the right Bol identity.
\end{ex}
\begin{prop}
Let $K$ be a group, $K^\prime$ the commutator subgroup of $K$, and $\varphi :K^\prime\to G$ a homomorphism. The Schreier extension $K\times_f^{\mathrm{Id}} G$ determined by the functions 
 \begin{equation}\label{fiszigma}f(\tau ,\sigma) = \varphi(\sigma^{-1}\tau^{-1}\sigma\tau),\; \tau, \sigma\in K, \quad\quad \Theta_\sigma=\mathrm{Id}, \; \sigma\in K\end{equation} 
has the following properties:
	\begin{enumerate}
	\item[(\emph{i)}] it is a right Bol loop,  
	\item[(\emph{ii)}] it is left nuclear if and only if the image $\varphi(K^\prime)$ of $K^\prime$ is contained in the centre of $G$, 
	\item[(\emph{iii)}] it is a group if and only if the homomorphism $\varphi$ is invariant under the conjugation of $K^\prime$ by elements of $K$.
\end{enumerate}
\end{prop}
\begin{proof} The elements $(\alpha,a)$, $(\beta,b)$ and $(\gamma,c)$ of $K\times_f^{\mathrm{Id}} G$ associate if and only if  
\begin{equation}\label{commhom} \varphi(\gamma^{-1}\beta^{-1}\alpha^{-1}\gamma\beta\alpha)a = \varphi(\gamma^{-1}\beta^{-1}\alpha^{-1}\beta\gamma\alpha)a\varphi(\gamma^{-1}\beta^{-1}\gamma\beta)
\end{equation}
for any $\alpha, \beta, \gamma\in K$ and $a\in G$. Consequently $K\times_f^{\mathrm{Id}} G$ is left nuclear if and only if  (\ref{commhom}) is an identity for $\alpha = \epsilon$, hence the assertion (ii) follows. Moreover $K\times_f^{\mathrm{Id}} G$ is a group if and only if it is left nuclear and the identity 
\[\varphi(\gamma^{-1}\beta^{-1}\alpha^{-1}\beta\gamma\alpha)^{-1}\varphi(\gamma^{-1}\beta^{-1}\alpha^{-1}\gamma\beta\alpha) = \varphi(\alpha^{-1}\gamma^{-1}\beta^{-1}\gamma\beta\alpha) = \varphi(\gamma^{-1}\beta^{-1}\gamma\beta)\]
holds. Hence $K\times_f^{\mathrm{Id}} G$ is a group if and only if $\varphi(\alpha^{-1}\xi\alpha) = \varphi(\xi)$ for any $\alpha, \xi\in K$, giving the  assertion (iii). \\ We have 
	\[(\gamma,c)\left[(\alpha,a)(\beta,b)\cdot(\alpha,a)\right] =(\gamma\alpha\beta\alpha,\varphi(\alpha^{-1}\beta^{-1}\alpha^{-1}\gamma^{-1}\alpha\beta\alpha\gamma)\;caba) =\]   
	\[= \left[(\gamma,c)(\alpha,a)\cdot(\beta,b)\right](\alpha,a),\] 
hence the assertion (i) is true.
\end{proof}
\begin{ex} \label{fId} Let $K$ be a group  with non-abelian commutator subgroup $K^\prime$ and let be $G = K^\prime$, $\varphi = \mathrm{Id}$. The Schreier extension $K\times_f^{\mathrm{Id}} K^\prime$ determined by (\ref{fiszigma})  is not nuclear and satisfies the right Bol loop identity.
\end{ex}
\begin{prop} Let $K$ and $G$ be groups and $\varphi :K\to G$ a homomorphism. The Schreier extension $K\times_e^{\Theta} G$ defined by the functions  
\begin{equation}\label{tetaszigma} f(\tau ,\sigma) = e,\quad \sigma, \tau\in K, \quad \Theta_\sigma = \iota_{\varphi(\sigma)},\quad \sigma\in K,\quad \iota_s(t) = sts^{-1}, \quad s,t\in G,\end{equation}  
is a right Bol loop. Then the following conditions are equivalent:
\begin{enumerate}
	\item[(i)] $K\times_e^{\Theta} G$ is left nuclear, 
	\item[(ii)] the image $\varphi(K^\prime)$ of the commutator subgroup $K^\prime\subset K$  is contained in the centre of $G$, 
	\item[(iii)] $K\times_f^{\mathrm{Id}} G$ is a group.
\end{enumerate}  
\end{prop}
\begin{proof} The elements $(\alpha,a), (\beta,b), (\gamma,c)\in K\times_e^{\Theta} G$ associate if and only if  
\begin{equation}\label{asseq}  \varphi(\gamma)^{-1}\varphi(\beta)^{-1}\varphi(\gamma)\varphi(\beta)a = a\varphi(\gamma)^{-1}\varphi(\beta)^{-1}\varphi(\gamma)\varphi(\beta),\quad \beta, \gamma\in K,\; a\in G.\end{equation}
Since these expressions are independent of $\alpha\in K$ the conditions $(i)$, $(ii)$ and $(iii)$ are equivalent. Moreover we have 
\[(\gamma,c)\left[(\alpha,a)(\beta,b)\cdot(\alpha,a)\right] = (\gamma\alpha\beta\alpha,\varphi(\alpha\beta\alpha)\;c\varphi(\alpha^{-1})a\varphi(\beta)^{-1}b\varphi(\alpha)^{-1}a)\]
and    
\[\left[(\gamma,c)(\alpha,a)\cdot(\beta,b)\right](\alpha,a) =  (\gamma\alpha\beta\alpha,\varphi(\alpha)\varphi(\beta)\varphi(\alpha)c\varphi(\alpha)^{-1}a\varphi(\beta)^{-1}b\varphi(\alpha)^{-1}a).\]
Hence $K\times_e^{\Theta} G$ is a right Bol loop.
\end{proof}
We obtain the following
\begin{ex} \label{eTheta} Let $K$ be a group with non-abelian commutator subgroup and let be $G = K$ and $\varphi = \mathrm{Id}$. The Schreier extension
$K\times_e^{\Theta} K$ determined by (\ref{tetaszigma}) is not nuclear and satisfies the right Bol identity.
\end{ex}

\subsection{Schreier decompositions}

\begin{defi} An Albert decomposition of a loop $L$ with respect to its middle and right nuclear normal subgroup $G$ is called a \emph{Schreier decomposition} of $L$. 
\end{defi}
For a data triple $(G,\kappa,\Sigma)$ of a Schreier decomposition of $L$ we have the expressions  
\begin{equation}\label{lTf} l_\sigma = \Sigma\cap\kappa(\sigma), \quad \Theta_{\sigma} = \mathrm{T}^{-1}_{l_\sigma} \quad \text{and}\quad f(\sigma,\tau) = 
l_{\sigma\tau}\backslash l_\sigma l_\tau ,\quad \sigma,\tau\in K, \end{equation}
where the maps $\mathrm{T}_x:G\to G$, $x\in L$, defined by $g\mapsto x\backslash gx$, are automorphisms of $G$. If $r\in G$ then $\mathrm{T}_r$ is the inner automorphism  $\iota_r(t) = rtr^{-1}$,\, $r,t\in G$.
\begin{lem}\label{decomp} The automorphisms $\mathrm{T}_{xr}$ and $\mathrm{T}_{rx}$ of $G$ with $x\in L$ and $r\in G$ can be decomposed as   
	\[\mathrm{T}_{xr} = \mathrm{T}_x\circ\iota_r, \quad \mathrm{T}_{rx} = \iota_r\circ\mathrm{T}_x.\] 
\end{lem}
\begin{proof} Since $s$ and $r$ belong to $\mathcal{N}_r(L)$, we have 
	\[\mathrm{T}_{xr}(s)\cdot xr = xr\cdot s = x\cdot\iota_r(s)r = x\iota_r(s)\cdot r =
	\mathrm{T}_x(\iota_r(s))x\cdot r = \mathrm{T}_x(\iota_r(s))(xr),\] 
	hence the first assertion is true. Similarly,  the second assertion follows from   
	\[\mathrm{T}_{rx}(s)\cdot rx = rx\cdot s = r\cdot xs  =  r\mathrm{T}_x(s)\cdot x = 
	\iota_r(\mathrm{T}_x(s))(rx),\]
	since $s\in \mathcal{N}_r(L)$ and $\mathrm{T}_x(s)$,  $r\in \mathcal{N}_m(L)$.
\end{proof}
\begin{cor} The map $\mathrm{T}:L\to \mathrm{Aut}(G)$ induces a map 
	\[\mathcal{A}: xG\mapsto \mathrm{T}_x\mathrm{Inn}(G): L/G\to\mathrm{Aut}(G)/\mathrm{Inn}(G),\]
	where $\mathrm{Inn}(G)$ denotes the group of inner automorphisms of $G$.
\end{cor}
\begin{thm} \label{image} The image of the map  $\mathrm{T}:L\to \mathrm{Aut}(G)$ consists of inner automorphisms if and only if there exists a left transversal $\Sigma$ of $L/G$ contained in the commutant $\mathcal{C}_L(G)$ of $G$. 
\end{thm}
\begin{proof} Assume that for any $x\in L$ the map $\mathrm{T}_x$ is  inner automorphism. Let $\Sigma$ be a left transversal of $L/G$ and  $g:\Sigma\to G$ a map satisfying $g(e) = e$ and $\mathrm{T}_x = \iota_{g(x)}$ 
	for any $x\in\Sigma$. Clearly, the set $\{x\cdot g(x)^{-1};\; x\in\Sigma\}\subset C_L(G)$ is a left transversal of $L/G$. According to Lemma \ref{decomp}, 
	$$\mathrm{T}_{x\cdot g(x)^{-1}} = \mathrm{T}_x\circ\iota_{g(x)^{-1}} = \iota_{g(x)}\circ\iota_{g(x)^{-1}} = \mathrm{Id}$$ 
	and hence $\Sigma^*\subset\mathcal{C}_L(G)$. Conversely, let $\Sigma$ be a left transversal of $L/G$ such that $\mathrm{T}_x = \mathrm{Id}_G$ for all $x\in\Sigma$. 
	Any element of $L$ is a product $x\cdot r$ with $x\in\Sigma$, $r\in G$ and hence Lemma \ref{decomp} yields that $\mathrm{T}_{x\cdot r} = \mathrm{T}_x\circ\iota_r = \iota_r$, 
	i.e. $\mathrm{T}_{x\cdot r}$ is  inner automorphism of $G$. 
\end{proof}
\begin{lem} \label{homonuc} The mapping $\mathrm{T}:L\to\mathrm{Aut}(G)$ is a homomorphism if and only if $G$ is nuclear.
\end{lem}
\begin{proof} For any $s\in G$, $x, y\in L$ we have $s\in \mathcal{N}_r(L)$,  $\mathrm{T}_y(s)\in \mathcal{N}_m(L)$ and hence 
	\[\mathrm{T}_{xy}(s)\cdot xy = x\cdot\mathrm{T}_y(s)y = x\mathrm{T}_y(s)\cdot y =
	\mathrm{T}_x(\mathrm{T}_y(s))x\cdot y.\] 
	It follows that $\mathrm{T}:L\to\mathrm{Aut}(G)$ is a homomorphism if and only if for any $x,y \in L,\; s\in G$ one has 
	$\mathrm{T}_x(\mathrm{T}_y(s))x\cdot y = \mathrm{T}_x(\mathrm{T}_y(s))\cdot xy$.  
	Since $\mathrm{T}_x,\; \mathrm{T}_y:G\to G$ are bijective maps, the map $\mathrm{T}:L\to\mathrm{Aut}(G)$ is a homomorphism if and only $G$ is left nuclear.
\end{proof}
It follows from Proposition \ref{NS} the following   
\begin{cor} The map $\mathrm{T}:L\to\mathrm{Aut}(G)$ is a homomorphism if $L$ satisfies one of the following conditions:\par
	(i)\;\; left inverse property,\par
	(ii)\; left alternative,\par
	(iii) flexible.
\end{cor}
According to Proposition 3.2.(i) in \cite{NS} the normal subgroup $\{\epsilon\}\times G$ of a Schreier extension $K\times_f^\Theta G$ is nuclear if and only if the maps $\Theta$ 
and $f$ satisfy the condition (\ref{lnuc}). Hence we have 
\begin{cor} For a Schreier extension $K\times_f^\Theta G$ the map $\mathrm{T}|_{\{\epsilon\}\times G} : K\times_f^\Theta G\to\mathrm{Aut}(\{\epsilon\}\times G)$ is a homomorphism if and only 
	if $K\times_f^\Theta G$ satisfies condition (\ref{lnuc}). 
\end{cor}

\subsection{Properties of Schreier decompositions}

\begin{thm} A loop $L$ has an automorphism-free Schreier decomposition with respect to a middle and right nuclear normal subgroup $G$ if and only if one of the following equivalent conditions is fulfilled:\\[1ex]
	(A)\; the image of the map $\mathrm{T}:L\to \mathrm{Aut}(G)$ consists of inner automorphisms, \\[1ex] 
	(B)\; there exists a left transversal of $L/G$ which is contained in the commutant $\mathcal{C}_L(G)$ of $G$ in $L$. (cf. Theorem 4. \cite{NSt2}).
\end{thm}
\begin{proof} Let $(G,\kappa,\Sigma)$ be a data triple for Schreier decomposition of $L$ and let $l:K\to L$ be the map $ l_\sigma = \Sigma\cap\kappa(\sigma)$. 
	The Schreier extension $\mathcal{L}(G,\kappa,\Sigma)$ corresponding to 
	$(G,\kappa,\Sigma)$ is automorphism-free if and only if 
	$\mathrm{T}_{l_\sigma} = \mathrm{Id}_G$ for any $\sigma\in K$, or equivalently, the 
	left transversal $\Sigma = \{l_\sigma;\; \sigma\in K\}$ of $L/G$ is contained in the commutant $\mathcal{C}_L(G)$ of $G$. Hence we obtain the assertion (B). 
	The equivalence of conditions (A) and (B) is proved in Theorem \ref{image}.  \end{proof}
\begin{thm} A loop $L$ has a factor-free Schreier decomposition with respect to a middle and right nuclear 
	normal subgroup $G$ if and only if $L$ contains a left transversal $\Sigma$ of $L/G$ which is a subloop of $L$ isomorphic to $L/G$.
\end{thm}
\begin{proof} Using the formula $ f(\sigma,\tau) = 
	l_{\sigma\tau}\backslash l_\sigma l_\tau$ in (\ref{lTf}) we obtain that the Schreier extension defined by (\ref{lTf}) is factor-free if and only if the map $l:K\to L$ satisfies 
	$l_\sigma l_\tau = l_{\sigma\tau}$ for any $\sigma,\tau\in K$, and hence  the map $l:K\to L$
	is a loop homomorphism. It follows that $L$ has a factor-free Schreier decomposition if and only if 
	there exists a left transversal $\Sigma$ of $L/G$ which is a subloop of $L$. \end{proof}
The following assertion shows the change of the Schreier decomposition of a loop $L$ with respect to a normal subgroup $G$, if we alter the underlying isomorphism.
\begin{prop} \label{change} Let $K\times_f^\Theta G$ be a Schreier decomposition of a loop $L$ with underlying isomorphism $\kappa :K\to L/G$ and let $\mu$ be 
	 automorphism of $K$. The Schreier extension $K\times_{\tilde{f}}^{\widetilde{\Theta}}G$ is a Schreier decomposition of $L$ with underlying isomorphism  
	$\kappa\circ\mu:K\to L/G$ if the functions $\widetilde{\Theta}:K\to\mathrm{Aut}(G)$, $\tilde{f}:K\times K\to G$ are expressed by 
	\begin{equation} \label{aut}\widetilde{\Theta}_\tau = \Theta_{\mu(\tau)},\;\tilde{f}(\sigma,\tau) = f(\mu(\sigma),\mu(\tau)).\end{equation}
\end{prop}
\begin{proof} We denote the multiplication of $K\times_{\tilde{f}}^{\widetilde{\Theta}} G$ by $\tilde{\cdot}$  and define the map 
	\[\mathcal{M}:(\sigma,s)\mapsto (\mu(\sigma),s): K\times G \to K\times G.\] 
	Since 
	\[\mathcal{M}((\sigma,s)\tilde{\cdot}(\tau,t)) = \mathcal{M}(\sigma\,\tau,\tilde{f}(\sigma\,\tau)\,\widetilde{\Theta}_\tau(s)\,t) = 
	\mathcal{M}(\sigma,s)\cdot\mathcal{M}(\tau,t),\] 
	the map $\mathcal{M}:K\times_{\tilde{f}}^{\widetilde{\Theta}}G\to K\times_f^\Theta G$  is  isomorphism inducing the identity on 
	the subgroup $\{\epsilon\}\times G$. It follows that $\mathcal{F}\circ\mathcal{M}:K\times_{\tilde{f}}^{\widetilde{\Theta}} G\to L$ 
	is  isomorphism extending the isomorphism  $\mathcal{I}:\{\epsilon\}\times G\to G$, $\mathcal{I}(\epsilon,t) = t$. Hence $K\times_{\tilde{f}}^{\widetilde{\Theta}} G$ is a 
	Schreier decomposition of $L$ with underlying isomorphism $\tilde{\kappa}:K\to L/G$ defined by \[\tilde{\kappa}(\sigma) = 
	\mathcal{F}\circ\mathcal{M}(\sigma,e)G, \quad \sigma\in K,\]  satisfying $\tilde{\kappa}(\sigma) = \kappa\circ\mu(\sigma)$. 
\end{proof}
Now, we investigate the change of a Schreier decomposition, if we alter the left transversal $\Sigma$ in the data triple $(G,\kappa,\Sigma)$.
\begin{thm} \label{cnlt}  Let $n:K\to G$ be a function with $n(\epsilon) = e$ and consider the left transversal 
	\begin{equation}\label{newlt}
		\bar{\Sigma} = \{l_\sigma n(\sigma)\in\kappa(\sigma),\; \sigma\in K\}\subset L.\end{equation}  
	The maps $\bar{\Theta}:K\to\mathrm{Aut}(G)$ and $\bar{f}:K\times K\to G$ of the Schreier extension  $K\times_{\bar{f}}^{\bar{\Theta}} G$ corresponding to the data triple  $(G,\kappa,\bar{\Sigma})$  for the Schreier decomposition of $L$ can be expressed as 
	       	\begin{equation} \label{tetaef} \bar{\Theta}_{\sigma} = \iota^{-1}_{n(\sigma)}\circ\Theta_{\sigma},\quad \bar{f}(\sigma,\tau) = 
		n(\sigma\tau)^{-1}f(\sigma,\tau) \Theta_{\tau}(n(\sigma))\,n(\tau).
	\end{equation} 
\end{thm}
\begin{proof}
	We compute the maps $\bar{\Theta}_{\sigma} = \mathrm{T}^{-1}_{\bar{l}_\sigma} \; \text{and}\; \bar{f}(\sigma,\tau) = 
	\bar{l}_{\sigma\tau}\backslash(\bar{l}_\sigma \bar{l}_\tau)$ \; corresponding to the left transversal (\ref{newlt}). Since 
	$ n(\sigma)$ belongs to the right and to the middle nucleus, the product $l_\sigma\,n(\sigma)\cdot \bar{\Theta}_{\sigma}(t) = 
	t\cdot l_\sigma\,n(\sigma)$ equals to 
	\[t\,l_\sigma\cdot n(\sigma) = l_\sigma\,\Theta_{\sigma}(t)\cdot n(\sigma) = l_\sigma\cdot\Theta_{\sigma}(t)\,n(\sigma) = 
	l_\sigma\, n(\sigma)\cdot n(\sigma)^{-1}\,\Theta_{\sigma}(t)\,n(\sigma). \]
	Hence 
	\begin{equation}\label{bartheta} \bar{\Theta}_{\sigma} = \iota^{-1}_{n(\sigma)}\circ\Theta_{\sigma}.\end{equation} 
	Similarly, using that $n(\sigma)$ belongs 
	to the middle and $n(\tau)$ to the right nucleus, the product 
	\begin{equation}\label{product}l_{\sigma\tau}\, n(\sigma\tau) \cdot \bar{f}(\sigma,\tau) =  l_{\sigma}\,n(\sigma)\cdot l_{\tau}\,n(\tau)\end{equation}  
	can be written as   
	\[l_{\sigma}\, \big(n(\sigma)\cdot l_{\tau}\,n(\tau)\big) = l_{\sigma}\, \big(n(\sigma)\, l_{\tau}\cdot n(\tau)\big) = 
	l_\sigma\,\big( l_\tau\cdot \Theta_{\tau}(n(\sigma))\,n(\tau)\big). \] 
	Since $\Theta_{\tau}(n(\sigma))\,n(\tau)$ is a right nuclear element, the expression (\ref{product}) equals to 
	$(l_\sigma l_\tau)\cdot \Theta_{\tau}(n(\sigma))\,n(\tau)$. Replacing $l_\sigma l_\tau = l_{\sigma\tau}\cdot\big( n(\sigma\tau) 
	\,n(\sigma\tau)^{-1}\cdot f(\sigma,\tau)\big)$ and using that $ n(\sigma\tau)$ and  $n(\sigma\tau)^{-1}$ belong to the middle nucleus, 
	we obtain for the expression (\ref{product}) that 
	\[l_{\sigma\tau}\, n(\sigma\tau) \cdot \bar{f}(\sigma,\tau) = l_{\sigma\tau}\, n(\sigma\tau) \cdot 
	\big(n(\sigma\tau)^{-1}f(\sigma,\tau) \Theta_{\tau}(n(\sigma))\,n(\tau)\big),\]
	and hence 
	\begin{equation}\label{barf}\bar{f}(\sigma,\tau) = n(\sigma\tau)^{-1}f(\sigma,\tau) \Theta_{\tau}(n(\sigma))\,n(\tau).\end{equation} 
	Using equations (\ref{bartheta}) and (\ref{barf}) we obtain the  assertion.
\end{proof} 

\subsection{Equivalent Schreier extensions}

Let $K$ be a loop, $G$ a group and let $K\times_f^\Theta G$ and  $K\times_{f^\prime}^{\Theta^\prime} G$ be Schreier extensions. 
\begin{defi} The Schreier extensions $K\times_f^\Theta G$ and  $K\times_{f^\prime}^{\Theta^\prime} G$ are called \emph{equivalent in a wider sense} if there exists an 
	extension $L$ of $G$ by $K$ such that $K\times_f^\Theta G$ and  $K\times_{f^\prime}^{\Theta^\prime} G$ are Schreier decompositions of $L$.
\end{defi}
\begin{lem} The Schreier extensions $K\times_f^\Theta G$ and  $K\times_{f^\prime}^{\Theta^\prime} G$ are equivalent in a wider sense if and only 
	if there exists an isomorphism $K\times_f^\Theta G\to K\times_{f^\prime}^{\Theta^\prime} G$ fixing all elements of $\{\epsilon\}\times G\subset K\times G$.\end{lem}
\begin{proof} Let $K\times_f^\Theta G$ and  $K\times_{f^\prime}^{\Theta^\prime} G$ be Schreier decompositions of a loop $L$ and let 
	$\mathcal{F}:K\times_f^\Theta G\to L$ and $\mathcal{F}^\prime:K\times_{f^\prime \Theta^\prime} G\to L$ be the isomorphisms extending the isomorphism $\mathcal{I}:\{\epsilon\}\times G\to G$ defined by $\mathcal{I}(\epsilon,t) = t$. Then the isomorphism $\mathcal{F}^{\prime -1}\circ\mathcal{F}: K\times_f^\Theta G\to K\times_{f^\prime}^{\Theta^\prime} G$ fixes all 
	elements of $\{\epsilon\}\times G$. Conversely, assume that $\psi:  K\times_f^\Theta G\to K\times_{f^\prime}^{\Theta^\prime} G$ is  isomorphism 
	fixing all elements of $\{\epsilon\}\times G$. According to Lemma \ref{exten} there is a loop $L$ and an isomorphisms $\mathcal{F}:K\times_f^\Theta G\to L$ 
	extending the isomorphism $\mathcal{I}:\{\epsilon\}\times G\to G$. Clearly, the isomorphism $\mathcal{F}\circ\psi^{-1}: K\times_{f^\prime}^{\Theta^\prime} G\to L$ 
	extends the isomorphism $\mathcal{I}:\{\epsilon\}\times G\to G$, hence $K\times_{f^\prime}^{\Theta^\prime} G$ is also a Schreier decomposition of $L$.
\end{proof}
\begin{defi} The Schreier extensions $K\times_f^\Theta G$ and  $K\times_{f^\prime}^{\Theta^\prime} G$ are called \emph{equivalent} if there exists an extension $L$ of $G$ 
	by $K$ such that $K\times_f^\Theta G$ and  $K\times_{f^\prime}^{\Theta^\prime} G$ are Schreier decompositions of $L$ with respect to  $G$ with the same underlying 
	isomorphism $\kappa:K\to L/G$.
\end{defi}
\begin{thm} \label{equi}Let $K\times_f^\Theta G$ and  $K\times_{f^\prime}^{\Theta^\prime} G$ be Schreier extensions. \\[1ex]
	\emph{(A)}\; \;$K\times_{f^\prime}^{\Theta^\prime} G$ is equivalent to $K\times_f^\Theta G$ if and only if there is a function $n:K\to G$ with $n(\epsilon) = e$ 
	such that $\Theta^\prime$ and $f^\prime$ are expressed by  
	\begin{equation} \label{ekviv}\Theta^\prime_{\sigma} = \iota^{-1}_{n(\sigma)}\circ\Theta_{\sigma},\quad\text{and}\quad f^\prime(\sigma,\tau) = 
	n(\sigma\tau)^{-1}f(\sigma,\tau) \Theta_{\tau}(n(\sigma))\,n(\tau).\end{equation}
	\emph{(B)}\; \;$K\times_{f^\prime}^{\Theta^\prime} G$ is equivalent in a wider sense to $K\times_f^\Theta G$  if and only if there is a function $n:K\to G$ with 
	$n(\epsilon) = e$, and an automorphism $\mu\in\mathrm{Aut}(K)$, such that $\Theta^\prime$ and $f^\prime$ are expressed by 
	\[\Theta^\prime_{\sigma} = \iota^{-1}_{n\circ\mu(\sigma)}\circ\Theta_{\mu(\sigma)},\]  
	and
	\[ f^\prime(\sigma,\tau) = 
	\big(n\circ\mu(\sigma\tau)\big)^{-1}f(\mu(\sigma),\mu(\tau)) \Theta_{\mu(\tau)}\big(n\circ\mu(\sigma)\big)\,n\circ\mu(\tau).\]
\end{thm}
\begin{proof} For the equivalent $K\times_f^\Theta G$ and  $K\times_{f^\prime}^{\Theta^\prime} G$ there exists a loop $L$ such that $K\times_f^\Theta G$ and  $K\times_{f^\prime}^{\Theta^\prime} G$ 
	are Schreier decompositions of $L$ with the same underlying isomorphism $\kappa:K\to L/G$. Hence the assertion (A) follows 
	from Theorem \ref{cnlt}. If $K\times_{f^\prime}^{\Theta^\prime} G$ is equivalent in a wider sense to $K\times_f^\Theta G$, then according to Proposition \ref{change} a change of the underlying isomorphism $K\to L/G$ of $L(T^\prime,f^\prime)$ by an automorphism $\mu$ of $K$ we obtain equivalent Schreier extensions. Hence the assertion (A) implies the assertion (B).
\end{proof} 
Now, we apply Theorem \ref{equi} (A) to Examples \ref{fId} and \ref{eTheta} in the case if $K$ is a perfect group, i.e. $K$ equals its own commutator subgroup. 
\begin{prop} 
Let $K$ be a perfect group. The Schreier extensions 
$K\times_f^{\mathrm{Id}} K$, respectively $K\times_e^{\Theta} K$  defined by 
\[f(\tau ,\sigma) = \sigma^{-1}\tau^{-1}\sigma\tau,\; \tau, \sigma\in K, \quad\quad \Theta_\sigma=\mathrm{Id}, \; \sigma\in K,\] 
respectively by 
\[f(\tau ,\sigma) = e,\; \sigma, \tau\in K, \quad \Theta_\sigma = \iota_{\sigma},\; \sigma\in K, \quad\text{where}\quad \iota_s(t) = sts^{-1}, \; s,t\in G,\]
are equivalent right Bol loops 
such that their normal subgroup $\{\epsilon\}\times G$ is not left nuclear.
\end{prop}
\begin{proof} We denote $f(\sigma,\tau) = \tau^{-1}\sigma^{-1}\tau\sigma$, $\Theta_\tau=\mathrm{Id}$  and $f^\prime(\sigma,\tau) = e$, $\Theta^\prime_\sigma = \iota_{\sigma}$ . Define $n(\sigma)=\sigma^{-1}$. Putting these functions into equations (\ref{ekviv}) we obtain the assertion.
\end{proof} 

\end{document}